\documentclass[10pt]{amsart}
\usepackage{amssymb}
\usepackage{amscd}

\newcommand{\bZ}{{\mathbb Z}}
\newcommand{\bR}{{\mathbb R}}

\newcommand{\bC}{{\mathbb C}}
\newcommand{\bF}{{\mathbb F}}
\newcommand{\bQ}{{\mathbb Q}}

\newcommand{\bT}{{\mathbb T}}

\newcommand{\bG}{{\mathbb G}}

\newcommand{\II}{{I_{\infty}^2}}

\newtheorem{thm}{Theorem}[section]
\newtheorem{lemma}[thm]{Lemma}
\newtheorem{cor}[thm]{Corollary}
\newtheorem{prop}[thm]{Proposition}

\numberwithin{equation}{section}

\begin{document}

\title[gamma filtration  of flag varieties ]{The gamma filtrations of $K$-theory of complete flag varieties }
 
\author{Nobuaki Yagita}

\address{ faculty of Education, 
Ibaraki University,
Mito, Ibaraki, Japan}
 
\email{nobuaki.yagita.math@vc.ibaraki.ac.jp, }


\maketitle

\begin{abstract}
Let $G$ be a compact Lie group and $T$ its maximal torus.
In this paper, we try to  compute $gr_{\gamma}^*(G/T)$
the graded ring associated  with the gamma filtration of
the complex $K$-theory $K^0(G/T)$. We use the Chow
ring of a corresponding versal flag variety.
\end{abstract}

\section{Introduction}

Let $p$ be a prime number. 
Let $K^*_{top}(X)$ be the complex $K$-theory localized at $p$ for a topological space $X$.  There are two typical
filtrations for $K^0_{top}(X)$ : the topological filtration defined by Atiyah \cite{At}
and the $\gamma$-filtration defined by Grothendieck \cite{Gr}, \cite{At}.
Let us write by $gr_{top}^*(X)$
and $gr_{\gamma}^*(X)$ the associated graded algebras for these filtrations.
 Let $G$ and  $T$ be a  connected compact Lie group
and its maximal torus. Then 
$ gr_{top}^*(G/T)\cong H^*(G/T)_{(p)}.$
However when $H^*(G)$ has $p$-torsion, it is not 
isomorphic to $gr_{\gamma}^*(G/T)$. In this paper, we try to study 
$gr_{\gamma}^*(G/T)$.

  To study the above (topological) $\gamma$-filtration, we use the corresponding algebraic $\gamma$-filtration. 
Given a field $k$ with $ch(k)=0$,
let $G_k$ and $T_k$ be a split  reductive group and split maximal torus over the field $k$,
corresponding to $G$ and  $T$.  Let $B_k$ be the Borel
subgroup containing $T_k$.
Let $\bG$ be a $G_k$-torsor.  Then $\bF=\bG/B_k$ is a
twisted form of the flag variety $G_k/B_k$. Hence we can consider
the $\gamma$-filtration and its graded ring  $gr^*_{\gamma}(\bF)$
for algebraic $K$-theory $K^0_{alg}(\bF)$ over $k$.
%The the geometric filtration 
%and the associated ring $gr_{geo}^*(\bF)$ are also defined from degree of Chow ring.

Moreover we take $k$
such that there is a $G_k$-torsor
$\bG$ which is isomorphic to a versal $G_k$-torsor ( for the definition of a versal $G_k$-torsor, see $\S 3$ below or
see \cite{Ga-Me-Se}, \cite{Tod}, \cite{Me-Ne-Za}, \cite{Ka1}).  Then  
$\bF=\bG/B_k$ is thought as the $most$ $twisted$  
complete flag variety.  (We say that such $\bF$ is
 a versal  flag variety  \cite{Me-Ne-Za}, \cite{Ka1})
We can see (from Panin \cite{Pa}
and Chevalley theorems)

\begin{thm}
Let $G$ be a compact simply connected Lie group.
Let $\bF=\bG/B_k$ be the versal flag variety. Then 
$K_{alg}^0(\bF)\cong K_{top}^0(G/T)$,   and
\[ gr_{\gamma}^*(G/T)\cong
gr_{\gamma}^*(\bF)\cong CH^*(\bF)/I
\quad for\ some \ ideal\ I\subset CH^*(\bF).\]
\end{thm} 

{\bf Remark.}  Karpenko conjectures that 
the above $I=0$
(\cite{Ka1} \cite{Ka2}).

Note that $gr_{\gamma}^*(G/T)$ is the associated ring of
the topological $K$-theory $K_{top}^0(G/T)$ but
$gr_{\gamma}^*(\bF)$ is that of the algebraic $K$-theory
$K^0_{alg}(\bF)$.  Hence the purely topological object $gr_{\gamma}^*(G/T)$ can be computed by a purely algebraic geometric object,
the Chow ring of a twisted flag variety $\bF$.

Let $BT$ be the classifying space of $T$.  From the
cofiber
sequence $G\to G/T\to BT$, we have the (modified)
Atiyah-Hirzebruch spectral sequence
\[ E_2^{*,*'}\cong H^{*,*'}(BT;K^*_{top}(G))
\Longrightarrow 
K^*_{top}(G/T).\]
From the above theorem, we can see  
\begin{cor}   We have 
$ E_{\infty}^{*,0}\cong gr_{\gamma}^*(G/T).$
\end{cor}

However, the above spectral sequence itself seems to be difficult to
compute.
In this paper, the proof of Theorem 1.1 is 
also given by the computation 
of each simple Lie group. 
 For example, in the following case, $gr_{\gamma}(G/T)/2\cong gr_{\gamma}(\bF)/2$
is computed (\cite{Ya6}).
\begin{thm}
Let $(G,p)$ be the following simply connected simple 
 group with $p$ torsion in $H^*(G)$.
Let $(G,p)=(Spin(n),2)$ for $n\le 10$, or
an exceptional Lie group except for $(E_7,2)$ and $(E_8,2,3)$.
Then
there are elements $b_s$ 
in $S(t)=\bZ[t_1,...,t_{\ell}]\cong H^*(BT)$ for $1\le s\le \ell$ such that
$ gr_{\gamma}(G/T)\cong CH^*(\bF)$, and
\[gr_{\gamma}(G/T)/p\cong
S(t)/(p,b_ib_j,b_k|1\le i,j\le2(p-1)<k\le \ell).\]
\end{thm}  

\begin{thm}
Let $(G,p)=(Spin(11),2)$. Then we can take
$c_1,c_i'$ in $S(t)=H^*(BT)\cong \bZ[t_1,...,t_5]$ 
for $2\le i\le 5$ with $|c_1|=2$, $|c_i'|=2i$  such that 
\[ gr_{\gamma}(G/T)/2\cong S(t)/(2, c_i'c_j', c_1^8c_i',c_1^{16}|2\le i\le j\le 5,\ (i,j)\not
=(2,4)).\]
\end{thm}

{\bf Remark.}  Quite recently,  Karpenko proves (
\cite{Ka2}) that
for $G=Spin(11)$ and $G=Spin(12)$,
we have
$gr_{\gamma}(\bF)\cong CH^*(\bF)$, and so the above
ring is isomorphic to  $CH^*(\bF)/2$.

The plan of this paper is the following.
In $\S 2$, we recall and prepare the topological arguments for $H^*(G/T)$, $K(1)^*(G/T)$
and $BP^*(G/T)$.  In $\S 3$, we recall the decomposition
of the motive of a versal flag variety
and recall the torsion index.
In $\S 4, \S 5$, we study the graded rings  of the $K$-theory (and
the Morava $K(1)$-theory).
In $\S 6$, $\S 7$, we study in the cases  $SO(m)$
and $Spin(m)$ for $p=2$.  
 In $\S 8-\S 11$, 
we study the cases that
$(G,p)$ are in Theorem 1.3 (e.g., $(E_8,5))$,
$(E_8,3)$, $(E_7,2)$ and $(E_8,2)$ respectively.

 \section{Lie groups $G$ and the  flag manifolds $G/T$}  

Let $G$ be a connected 
 compact Lie group.
By Borel, its $mod(p)$ cohomology is (for $p$ odd)
\[ (2.1)\quad H^*(G;\bZ/p)\cong P(y)/p\otimes \Lambda(x_1,...,x_{\ell}),
\quad \ell=rank(G)\]
\[ \quad  with  \ \  P(y)=\bZ_{(p)}[y_1,...,y_k]/(y_1^{p^{r_1}},...,
y_k^{p^{r_k}})
\]
where the degree $|y_i|$ of $y_i$ is even and 
$|x_j|$ is odd.  When $p=2$, a graded ring $grH^*(G;\bZ/2)$ is isomorphic to the
right  hand side ring, e.g. $x_j^2=y_{i_j}$ for some $y_{i_j}$.
In this paper, $H^*(G;\bZ/2)$ means this $grH^*(G;\bZ/2)$
so that $(2.1)$ is satisfied also for $p=2$.

 Let $T$ be the  maximal torus of $G$.  and $BT$ be the classifying space of $T$.
We consider the fibering (\cite{Tod}, \cite{Mi-Ni})
$ G\stackrel{\pi}{\to}G/T\stackrel{i}{\to}BT$
and the induced spectral sequence 
\[ E_2^{*,*'}=H^*(BT;H^{*'}(G;\bZ/p)) \Longrightarrow H^*(G/T;\bZ/p).\] 
The cohomology of the classifying space of the torus is  given by
$H^*(BT)\cong S(t)=\bZ[t_1,...,t_{\ell}]$ with $|t_i|=2$,
where $t_i=pr_i^*(c_1)$ is the $1$-st  Chern class
induced from
$ T=S^1\times ...\times S^1\stackrel{pr_i}{\to}S^1\subset U(1)$
for the $i$-th projection $pr_i$.
Note that $\ell=rank(G)$ is also the number of the odd degree generators $x_i$ in
$H^*(G;\bZ/p)$.  

It is well
known that $y_i$ are permanent cycles and 
that there is a regular sequence (\cite{Tod}, \cite{Mi-Ni})
$(\bar b_1,...,\bar b_{\ell})$ in $H^*(BT)/(p)$ such that $d_{|x_i|+1}(x_i)=\bar b_i$ (this $\bar b_i$ is called the transgressive element).  Thus we get
\[ E_{\infty}^{*,*'}\cong grH^*(G/T;\bZ/p)\cong P(y)/p\otimes 
S(t)/(\bar b_1,...,\bar b_{\ell}).\]

 Moreover we know that $G/T$ is a manifold 
such that $H^*(G/T)$ is torsion free, and 
\[(2.2)\quad H^*(G/T)_{(p)}\cong \bZ_{(p)}[y_1,..,y_k]\otimes S(t)/(f_1,...,f_k,b_1,...,b_{\ell})\]
where $b_i=\bar b_i\ mod(p)$ and $f_i=y_i^{p^{r_i}}\ mod(t_1,...,t_{\ell}).$

Let $BP^*(-)$ be the Brown-Peterson theory
with the coefficients ring $BP^*\cong \bZ_{(p)}[v_1,v_2,...],$
$|v_i|=-2(p^i-1)$ (\cite{Ha}, \cite{Ra}).
Since $H^*(G/T)$ is torsion free,  the 
Atiyah-Hirzebruch 
spectral sequence (AHss) collapses.  Hence  we also know 
\[(2.3)\quad BP^*(G/T)\cong BP^*[y_1,...,y_k]\otimes S(t)/(\tilde f_1,...,\tilde f_k,
\tilde b_1,...,\tilde b_{\ell})\]
where $\tilde b_i=b_i\ mod(BP^{<0})$ and $\tilde f_i=f_i\ mod(BP^{<0}).$

Recall the  Morava $K$-theory
$K(n)^*(X)$ with the coefficient ring $K(n)^*=\bZ/p[v_n,v_n^{-1}]
$.
Similarly we can define connected or integral Morava
K-theories with 
\[ k(n)^*=\bZ/p[v_n],\ \ \tilde k(n)^*=\bZ_{(p)}[v_n],
\ \ \tilde K(n)^*=\bZ_{(p)}[v_n,v_n^{-1}]. \]
It is known (as the Conner-Floyd type theorem)
\[\tilde K(1)^*(X)\cong  BP^*(X)\otimes _{BP^*}\tilde K(1)^*.\]
The above fact does not hold for $K(n)$-theory with 
$n\ge 2$.

Here we consider the connected Morava K-theory $k(n)^*(X)$
(such that $K(n)^*(X)\cong k(n)^*(X)[v_n^{-1}]$) 
and the Thom natural homomorphism
 $\rho : k(n)^*(X)\to H^*(X;\bZ/p)$.
 Recall that there is an exact sequence
(Sullivan exact sequence \cite{Ra}, \cite{Ya2})
\[...\to k(n)^{*+2(p^n-1)}(X)\stackrel{v_n}{\to} k(n)^*(X)\stackrel{\rho}{\to}
H^*(X;\bZ/p)\stackrel{\delta}{\to}...\]
such that $\rho\cdot \delta(X)=Q_n(X)$.
Here the Milnor $Q_i$ operation
\[ Q_i: H^*(X;\bZ/p)\to H^{*+2p^i-1}(X;\bZ/p)\]
is defined by $Q_0=\beta$ and $Q_{i+1}=[P^{p^i}Q_i,Q_iP^{p^i}]$
for the Bockstein operation $\beta$ and the reduced power operation $P^j$.

We consider  the Serre spectral sequence  
\[ E_2^{*,*'} \cong H^*(B;H^*(F;\bZ/p))
\Longrightarrow  H^*(E;\bZ/p). \]
induced from the fibering $F\stackrel{i}{\to} E\stackrel{\pi}{\to}B$ with $H^*(B)\cong H^{even}(B)$.
By using the Sullivan exact sequence, we can prove ;
\begin{lemma} (Lemma 4.3 in \cite{Ya2})
In the spectral sequence $E_r^{*,*'}$ above,
suppose that  there is $x\in H^*(F;\bZ/p)$ such that
\[ (*)\quad  y=Q_n(x)\not=0 \quad  and \quad
b= d_{|x|+1}(x)\not = 0\in E_{|x+1}^{*,0}.\]
Moreover suppose that  $E_{|x|+1}^{0,|x|}
\cong \bZ/p\{x\}\cong \bZ/p$.
%if $ Q_n(\tilde  \bF)=y,$ then
%\[ (**)\quad d_r(\tilde \bF)\not =0\ \  for \ r<|\bF|+1,\quad   %or\quad   
%d_{|\bF|+1}(\tilde \bF)=b. \]
Then there are $y'\in k(n)^*(E)$ and $b'\in k(n)^*(B)$ such that
$i^*(y')=y$, $\rho(b')=b$ and that
\[(**)\quad v_ny'=\lambda \pi^*(b')\quad  in \ \  k(n)^*(E),
\ \  for \ \lambda\not =0\in \bZ/p.\]
Conversely if $(**)$ holds in $k(n)^*(E)$ for $y=i^*(y')\not =0$ and $b=\rho(b')\not =0$, then there is $x\in H^*(F;\bZ/p)$ such that $(*)$ holds.
\end{lemma}

{\bf Remark.} (Remark 4.8 in \cite{Ya2})
The above lemma also holds letting
$k(0)^*(X)=H^*(X;\bZ)$ and $v_0=p$. 
This fact is well known (Lemma 2.1 in \cite{Tod}).

\begin{cor} 
Let $b\not =0$ be the transgression image of $x$, i.e.
$d_{|x|+1}(x)=b\in H^*(G/T)/p$.  Then there is a lift $b\in BP^*(BT)\cong BP^*\otimes S(t)$ of $b\in S(t)/p$ such that
\[ b=\sum_{i=0} v_iy(i) \in \ BP^*(G/T)/\II
\quad (i.e., \ b=v_iy(i)\in\ k(i)^*(G/T))/(v_i^2))\]
where $y(i)\in H^*(G/T;\bZ/p)$ with $\pi^*y(i)=Q_ix$.
\end{cor}

\section{versal flag varieties}

 Let $G_k$ be the split reductive algebraic  group corresponding to $G$, and $T_k$ be the split maximal
torus corresponding to $T$.  Let $B_k$ be the Borel subgroup
with $T_k\subset B_k$.   Note that $G_k/B_k$ is cellular, and
$CH^*(G_k/T_k)\cong CH^*(G_k/B_k)$.
Hence we have   
\[ CH^*(G_k/B_k)\cong H^*(G/T)\quad  and 
\quad  CH^*(BB_k)\cong H^*(BT).
\]

Let us write by $\Omega^*(X)$  the $BP$-version of the algebraic cobordism defined by Levine-Morel ([
\cite{Le-Mo1},\cite{Le-Mo2},\cite{Ya2},\cite{Ya4})  such that
\[ \Omega^*(X)=MGL^{2*,*}(X)_{(p)}\otimes_{MU_{(p)}^*}BP^*,
\quad \Omega^*(X)\otimes _{BP^*}\bZ_{(p)}\cong CH^*(X)_{(p)}\] where $MGL^{*,*'}(X)$ is the algebraic cobordism theory
defined by Voevodsky with $MGL^{2*,*}(pt.)\cong 
MU^*$ the complex cobordism ring.
There is a natural (realization) map $\Omega^*(X)\to
BP^*(X(\bC))$.  In particular, we have 
$\Omega^*(G_k/B_k)\cong BP^*(G/T).$
Let $I_n=(p,v_1,...,v_{n-1})$ and $I_{\infty}=(p,v_1,...)$
be the (prime invariant) ideals in $BP^*$.  We also note
\[ \Omega^*(G_k/B_k)/I_{\infty}\cong
BP^*(G/T)/I_{\infty}\cong H^*(G/T)/p.\]

Let  $\bG$ be  a  nontrivial $G_k$-torsor.
We can construct a twisted form of $G_k/B_k$ by
$(\bG\times G_k/B_k)/G_k\cong \bG/B_k.$
We will study  the twisted flag variety $\bF=\bG/B_k$.

By extending the arguments by Vishik \cite{Vi} for quadrics to that for  flag varieties, Petrov, Semenov and Zainoulline define the 
$J$-invariant of $\bG$. Recall the expression (2.1) in $\S 2$ 
\[(*)\quad H^*(G;\bZ/p)\cong \bZ/p[y_1,...,y_s]/(y_1^{p^{r_1}},...,y_s^{p^{r_s}})
\otimes \Lambda(x_1,...,x_{\ell}).
\]
Roughly speaking (for the detailed definition, see
\cite{Pe-Se-Za}),
the $J$-invariant is defined as 
$ J_p(\bG)=(j_1,...,j_{s})$
if $j_i$ is the minimal integer such that
\[ y_i^{p^{j_i}}\in Im(res_{CH})\quad mod (y_1,...,y_{i-1}, t_1,...,t_{\ell})\]
for $res_{CH}:CH^*(\bF)\to CH^*(\bar \bF)$.
Here  we take $|y_1|\le |y_2|\le...$ in $(*)$.  Hence $0\le j_i\le r_i$ and 
$J_p(\bG)=(0,...,0)$ if and only if $\bG$ split by an extension of the index coprime to $p$.
One of the main results in \cite{Pe-Se-Za} is 
\begin{thm} (Theorem 5.13 in \cite{Pe-Se-Za}
and Theorem 4.3 in \cite{Se-Zh})
Let $\bG$ be a $G_k$-torsor over $k$, $\bF=\bG/B_k$
 and $J_p(\bG)=(j_1,...,j_{s})$.
Then there is a $p$-localized  motive $R(\bG)$ such that
\[M(\bF)_{(p)}\cong  \oplus _u R(\bG)\otimes \bT^{\otimes u}.\]
Here $\bT^{\otimes u}$ are Tate motives with 
$CH^*(\oplus_u \bT^{\otimes u})/p\cong  P'(y)\otimes S(t)/(b)$ 
where \[P'(y)=\bZ/p[y_1^{p^{j_1}},...,y_{s}^{p^{j_{s}}}]/(y_1^{p^{r_1}},...,y_{s}^{p^{r_{s}}}
)\subset P(y)/p,\]
\[\quad S(t)/(b)=S(t)/(b_1,...,b_{\ell}).\]
The $mod(p)$ Chow group of $\bar R(
\bG)=R(\bG)\otimes \bar k$  is  given by 
\[ CH^*(\bar R(\bG))/p\cong \bZ/p[y_1,...,y_s]/(y_1^{p^{j_1}},...,y_{s}^{p^{j_{s}}}).
 \]
Hence we have $ CH^*(\bar \bF)/p\cong  CH^*( \bar R(\bG))\otimes
P'(y)\otimes S(t)/(b) $ and \\
\[  CH^*(\bF)/p\cong 
CH^*( R(\bG))\otimes P'(y)\otimes S(t)/(b).\]
\end{thm}
Let $P_k$ be special (namely, any extension is split,
e.g. $B_k$). 
Let us consider an embedding of $G_k$ into the general linear group $GL_N$ for some $N$.  This makes $GL_N$ a $G_k$-torsor over the quotient variety $S=GL_N/G_k$.
Let $F$ be the function field $k(S)$ and  define
the $versal$ $G_k$-$torsor$ $E$ to be the $G_k$-torsor over $F$ given by the generic fiber of $GL_N\to S$. 
(For details, see \cite{Ga-Me-Se}, \cite{To2}, \cite{Me-Ne-Za}, \cite{Ka1}.) 
\[\begin{CD}
        E@>>> GL_N\\
       @VVV     @VVV \\
     Spec(k(S))  @>>> S=GL_N/G_k
\end{CD}\]
The corresponding flag variety $E/P_k$ is called $generically$ $twisted$ or $versal$ flag
variety, which is considered as the most complicated twisted
flag variety (for given $G_k,P_k$). It is known that the Chow ring
$CH^*(E/P_k)$ is not dependent to the choice
of  generic $G_k$-torsors $E$ (Remark 2.3 in \cite{Ka1}).

Karpenko and Merkurjev proved the following result
 for a versal (generically twisted) flag variety. 
\begin{thm}
(Karpenko Lemma 2.1 in \cite{Ka1})
Let $h^*(X)$ be an oriented cohomology theory
(e.g., $CH^*(X)$, $\Omega^*(X)$).
Let $P_k$ be a parabolic subgroup of $G_k$ and $\bG/P_k$ be a versal flag variety.
Then the natural map
$h^*(BP_k)\to h^*(\bG/P_k)$ is surjective
for the classifying space $BP_k$ for $P_k$-bundle.
\end{thm}
\begin{cor}
If $\bG$ is versal, then $CH^*(\bF)=CH^*(\bG/B_k)$
is multiplicatively generated by elements $t_i$ in $S(t)$. 
\end{cor}
\begin{cor} If  $\bG$ is 
 versal, then  $J(\bG)=(r_1,...,r_s)$, i.e. $r_i=j_i$.
Hence $P'(y)=\bZ/p$.
\end{cor}
\begin{proof}
If $j_i<r_i$, then $0\not =y_i^{p^{j_i}}\in res(CH^*(\bF)\to
CH^*(\bar \bF))$, which is in the image from $S(t)$
by the preceding theorem.
This contradicts to $CH^*(\bar \bF;\bZ/p)\cong
P(y)/p\otimes S(t)/(b)$ and $0\not =y_i^{p^{j_i}}\in P(y)/p$.
\end{proof}

Hence we have surjections for a versal variety $\bF$
\[ CH^*(BB_k)\twoheadrightarrow  CH^*(\bF)\stackrel{pr.}{\twoheadrightarrow} CH^*(R(\bG)).\]

{\bf Remark.}  In this paper, a map $A\to B$ (resp. $A\cong B$) for rings $A,B$
means a ring map (resp. a ring isomorphism).
  However $CH^*(R(\bG))$ does not have a
 canonical ring structure.
 Hence a map $A\to CH^*(R(\bG))/p$ (resp. $A\cong CH^*(R(\bG))/p$)
means only a (graded) additive map 
(resp. additive isomorphism).

We study in \cite{Ya6}
what elements in $CH^*(BB_k)$ 
generate $CH^*(R(\bG))$.
By giving the filtration on $S(t)$ by $b_i$, we 
can write 
\[gr S(t)/p\cong A\otimes S(t)/(b_1,...,b_{\ell})\quad
for \ A=\bZ/p[b_1,...,b_{\ell}].\]
In particular, we have maps
$ A\stackrel{i_A}{\to} CH^*(\bF)/p\to CH^*(R(\bG))/p.$
We also see that
the above composition map is surjective
(see also Lemma 3.7  below).
\begin{lemma}
Suppose that there are $f_1(b),...,f_s(b)\in A$ such that 
$CH^*(R(\bG))/p\cong A/(f_1(b),...,f_s(b))$. Moreover if $f_i(b)=0$
also in $CH^*(\bF)/p$, we have the isomorphism
\[CH^*(\bF)/p\cong S(t)/(p, f_1(b),...,f_s(b)).\]
\end{lemma}\begin{proof}
Using  $(f_1,...,f_s)\subset (b)=(b_1,...,b_{\ell}))$,
we have additively
\[S(t)/(f_1,...,f_s)\cong (A\otimes S(t)/(b))/(f_1,...,f_s)\]
\[\cong A/(f_1,...,f_s)\otimes S(t)/(b)\]
\[\cong CH^*(\bR(\bG))/p\otimes S(t)/(b)\cong
CH^*(\bF)/p.\]
Of course there is a ring surjective map 
$S(f)/(f_1,...,f_s)\to CH^*(\bF)/p$, this map must be isomorphic.
\end{proof}

Let $dim(G/T)=2d$.  Then the torsion index is defined as
\[ t(G)=|H^{2d}(G/T;\bZ)/H^{2d}(BT;\bZ)|.\]
Let $n(\bG)$ be the greatest common divisor of the degrees of all finite field extension $k'$ of $k$ such that $\bG$ 
becomes trivial over $k'$.  Then by Grothendieck 
\cite{Gr}, it is known that $n(\bG)$ divides $t(G)$.  Moreover, $\bG$ is versal, then $n(\bG)=t(G)$   (\cite{To2}, \cite{Me-Ne-Za}, \cite{Ka1}).
Note that $t(G_1\times G_2)=t(G_1)t(G_2)$.
It is well known that  if $H^*(G)$ has a $p$-torsion, then
$p$ divides the torsion index $t(G)$. Torsion index for
simply connected compact Lie groups are completely determined by Totaro \cite{To1}, \cite{To2}.

For $N>0$, let us  write \ \ 
$A_N=\bZ/p\{b_{i_1}...b_{i_k}| |b_{i_1}|+...+|b_{i_k}|\le N\}.$
\begin{lemma}
Let  $b\in Ker(pr)$ for $pr:A_{2d}\twoheadrightarrow CH^*(R(\bG))$.  Then
\[ b=\sum b'u'\quad with \ b'\in A_{2d},\ \ u'\in S(t)^+/(p,b),\ i.e., \ |u'|>0.\]
\end{lemma}
Let us write 
\[  y_{top}=\Pi_{i=1}^s y_i^{p^{r_i}-1}\quad (reps.\ t_{top})\]
the generator of the highest  degree 
in $P(y)$ (resp. $S(t)/(b)$) so that $f=y_{top}t_{top}$
is the fundamental class in $H^{2d}(G/T)$.
\begin{lemma}  The following map is surjective
\[ A_N\twoheadrightarrow CH^*(R(\bG))/p\quad where\ 
N=|y_{top}|.\]
\end{lemma}
\begin{proof}  In the preceding lemma,  $A_{N}\otimes u$ for $|u|>0$ maps zero in $CH^*(R(\bG))/p$.  Since each element
in $S(t)$ is written by an element in $A_N\otimes S(t)/(b)$,
we have the lemma.
\end{proof}
\begin{cor}
If  $b_i\not=0$ in $CH^*(X)/p$,  then  so in $CH^*(R(\bG_k))/p$.
\end{cor}
\begin{proof}
Let $pr(b_i)=0$.  From Lemma 3.6, 
 $b_i=\sum b'u'$ for $|u'|>0$, and hence $b'\in Ideal(b_1,...,b_{i-1})$.
This contradict to that $(b_1,...,b_{\ell})$ is regular.
\end{proof}

Now we consider the torsion index $t(G)$.
\begin{lemma} (\cite{Ya6})
Let $\tilde b=b_{i_1}... b_{i_k}$ in
$S(t)$  such that
in $H^*(G/T)_{(p)}$
\[ \tilde b=p^s(y_{top}+\sum yt),\quad
|t|>0\]
for some $y\in P(y)$  and $t\in S(t)$.
Then $t(G)_{(p)}\le p^s$. 
%and $t(G)_{(p)}$
%is the smallest $p^s$ satisfied above equation.
\end{lemma}

\section{filtrations of $K$-theories}

We first recall the topological filtration defined by Atiyah.
Let $X$ be a topological space.  Let $K^*_{top}(X)$ be the complex $K$-theory ; the Grothendieck group generated by
complex bundles over $X$.  Let $X^i$ be an $i$-dimensional skeleton of $X$.
Define the topological filtration of $K^*_{top}(X)$ by
$ F_{top}^i(X)=Ker(K^*_{top}(X)\to K^*_{top}(X^i))$
and the associated graded algebra
$  gr_{top}^i(X)=F_{top}^i(X)/F_{top}^{i+1}(X).$
Consider the  AHss
 \[ E_2^{*,*'}(X)\cong H^*(X)\otimes K^{*'}_{top}\Longrightarrow K^*(X)_{top}.\]
 By the construction of the spectral sequence, we have
 \begin{lemma}(Atiyah \cite{At})  $gr^*_{top}(X)
\cong E_{\infty}^{*,0}(X)$.
 \end{lemma}
Note that the above isomorphism 
\[ E_{\infty}^{2*,0}\cong gr^{2*}(K^*_{top}(X))
\to gr^{2*}(K^0_{top}(X))\cong gr_{top}^{2*}(X)\]
is given by $x\mapsto B^*x$ for the Bott periodicity, i.e.,
$K_{top}^*\cong \bZ[B.B^{-1}]$ and $deg(B)=(-2,-1)$.
 
 Next we consider the geometric filtration.
 Let $X$ be a smooth algebraic variety over a subfield  $k$ of $\bC$.
 Let $K_{alg}^0(X)$ be the algebraic $K$-theory which is the Grothendieck group
 generated by $k$-vector bundles over $X$.  It is also isomorphic to the
Grothendieck group genrated by coherent sheaves over $X$ (we assumed $X$ smooth).
This $K$ -theory can be written by the motivic $K$-theory $AK^{*,*'}(X)$ (\cite{Vo1}, \cite{Vo2},
\cite{Ya3}), i.e.,
\[K_{alg}^i(X)\cong \oplus _{j} AK^{2j-i,j}(X).\]
In particular $K_{alg}^0(X)\cong 
AK^{2*,*}(X)=\oplus_jAK^{2j,j}(X)$.  

The geometric filtration (\cite{Gr}) is defined as
\[ F^{2i}_{geo}(X)=\{[O_V]| codim_XV\ge i\}\]
(and $F^{2i-1}_{geo}(X)=F^{2i}_{geo}(X)$) 
where $O_V$ is the structural sheaf of closed subvariety $V$ of $X$.
 Here we recall the motivic AHss  \cite{Ya3}, \cite{Ya4}
 \[ AE_2^{*,*',*''}(X)\cong H^{*,*'}(X;K^{*''})\Longrightarrow AK^{*,*'}(X)\]
where $K^*=AK^{2*,*}(pt.).$
  Note that
\[ AE_2^{2*,*,*''}(X)\cong H^{2*,*}(X;K^{*''})\cong CH^*(X)\otimes K^{*''}.\]
Hence 
$AE_{\infty}^{2*,*,0}(X)$ is a quotient of $ CH^*(X)$ by dimensional reason of degree of differential
$d_r$ (i.e., $d_rAE_{r}^{2*,*,*''}(X)=0$).  Thus we have 
 \begin{lemma}  (Lemma 6.2 in \cite{Ya5})  We have
\[gr^{2*}_{geo}(X)\cong AE_{\infty}^{2*,*,0}(X)\cong CH^*(X)/I\]
where $I=\cup_rIm(d_r)$.
 \end{lemma}
  \begin{lemma}  (\cite{Ya5})   Let $t_{\bC}: K_{alg}^0(X)\to K_{top}^0(X(\bC))$
be the realization map.  Then $F_{geo}^i(X)\subset (t_{\bC}^*)^{-1}F_{top}^i(X(\bC)).$
\end{lemma}

At last, we consider the gamma filtration.
Let $\lambda^i(x)$ be the exterior power of the vector bundle
$x\in K_{alg}^0(X)$ and $\lambda_t(x)=\sum \lambda^i(x)t^i$.  Let us denote 
\[ \lambda_{t/(1-t)}(x)=\gamma_t(x)=\sum \gamma^i(x)t^i
\quad (i.e.\ \gamma^n(x)=\lambda^n(x+n-1)).\]
The gamma filtration is defined as
\[F_{\gamma}^{2i}(X)=\{\gamma^{i_1}(x_1)\cdot ...\cdot \gamma^{i_m}(x_m)|
  i_1+...+i_m\ge i, x_j\in K_{alg}^0(X)\}.\]
  Then we can see $F_{\gamma}^i(X)\subset F_{geo}^i(X)$. (Similarly we can define $F_{\gamma}(X)$ for a topological space $X$.)  In Proposition 12.5 in [At],
  Atiyah proved $F_{\gamma}^i(X)\subset F_{top}^i(X)$ in $K_{top}(X)$.
  Moreover Atiyah's  arguments work also in $K_{alg}^0(X)$ and this fact is well known (
\cite{Ga-Za}, \cite{Ju}, \cite{Ya5}).
  Let $\epsilon: K^0_{alg}(X)\to \bZ$ be the augmentation map
and $c_i(x)\in H^{2i,i}(X)$ the Chern class.  
Let 
$q:CH^*(X)\cong E_{2}^{2*,*,0} \twoheadrightarrow  E_{\infty}^{2*,*,0}$
 be the quotient map.
Then (see p. 63 in \cite{At})
we have
\[ q (c_n(x))=[\gamma^n(x-\epsilon(x))].\]
\begin{lemma} (Atiyah \cite{At}, \cite{Ya5})
The condition $F_{\gamma}^{2*}(X)=F_{top}^{2*}(X)$ 
(resp. $F_{\gamma}^{2*}(X)=F_{geo}^{2*}(X)$)
 is equivalent to that $E_{\infty}^{2*,0}(X)$ (resp. $AE_{\infty}^{2*,*,0}(X)$)
is (multiplicatively) generated by 
 Chern classes in $H^{2*}(X)$ (resp. $CH^*(X)$).
 \end{lemma}
By Conner-Floyd type theorem
(\cite{Ra}, \cite{Ha}), it is well known 
\[AK^{*,*'}(X)\cong (MGL^{*,*'}(X)\otimes_{MU^*}\bZ)\otimes \bZ[B,B^{-1}]\]
where the $MU^*$-module structure of $\bZ$ is given by the Todd genus,
and $B$ is the Bott periodicity.
Since the Todd genus of $v_1$ (resp. $v_i$, $i>1$) is $1$ (resp. $0$),
we can write
\[ AK^{*,*'}(X)\cong ABP^{*,*'}(X)\otimes _{BP^*}\bZ[B,B^{-1}]\quad with\ B^{p-1}=v_1.\]
Recall that $A\tilde K(1)^{*,*'}(X)$ is  the algebraic Morava $K$-theory with $\tilde K((1)^{2*}=\bZ_{(p)}[v_1,v_1^{-1}].$  
By the Landweber exact functor theorem (see \cite{Ya3}),
we have 
$ A\tilde K(1)^{*,*'}(X)\cong ABP^{*,*'}(X)\otimes _{BP^*}
\tilde K(1)^*.$
Thus we have 
\begin{lemma}  (\cite{Ya5})  There is a natural isomorphism
\[A\tilde K^{*,*'}(X)\cong A\tilde K(1)^{*,*'}(X)\otimes _{\tilde K(1)^*}\bZ[B,B^{-1}]\quad
with \ v_1=B^{p-1}.\]
\end{lemma}  
\begin{lemma} (\cite{Ya5})  
Let  $E(AK)_r$ (resp. $E(A\tilde K(1))_r$) be 
the AHss converging to $AK^{*,*'}(X)$ (resp.
$A\tilde K(1)^*(X)$.  Then
\[ E(AK)_r^{*,*',*''}\cong E(A\tilde K(1)_r^{*,*',*''}\otimes _{\tilde K(1)^*}\bZ_{(p)}[B,B^{-1}].\]
In particular, $gr_{geo}^*(X)\cong E(AK)^{2*,*,0}_{\infty}
\cong E(\tilde AK(1))^{2*,*,0}_{\infty}.$
\end{lemma}

From the above lemmas,  it is sufficient to consider the Morava $K$-theory
$A\tilde K(1)^{*,*'}(X) $ when we want to study $AK^{*,*'}(X)$.
For ease of notations, let us write simply 
\[K^{2*}(X)=A\tilde K(1)^{2*,*}(X),\quad so\ that \ \  K^*(pt.)\cong \bZ_{(p)}[v_1,v_1^{-1}],\]
\[k^{2*}(X)=A\tilde k(1)^{2*,*}(X),\quad so\ that \ \  k^*(pt.)\cong \bZ_{(p)}[v_1].\]
Hereafter of this paper, we only consider this 
Morava $K$-theory
$K^{2*}(X)$  instead of
$AK^{2*,*}(X)$ or $K_{alg}^0(X)$.

%\begin{lemma}
%If an element $y\in K^{*}(X)$ is represented by 
%$0\not =y'$  (resp. $y'',y'''$) in $ gr^{i}_{\gamma}(X)$ (resp.
%$gr^{j}_{geo}(X), gr_{top}^k(X(\bC))$), then
%\[ i\le j\le k,\quad and \quad i=k=j\ mod(2(p-1)).\]
%\end{lemma}
%\begin{proof}
%The element $y$ is represented
%as $v_1^{s}y' \in K^{*}(X)/F_{\gamma}^{2i+1}$, and
%$v_1^{t}y''\in K^{*}(X)/F_{geo}^{2j+1}$
%for $s,t\in \bZ$. So $|y''|-|y'|=2(t-s)(p-1)$.
%\end{proof}

\section{ The restriction map for the $K$-theory}
 
We consider the restriction maps  that
\[  res_K: K^*(X)\to K^*(\bar X).\]
By Panin \cite{Pa}, it is known that $K_{alg}^0(\bF)$ is torsion free
for each  twisted flag varieties $\bF=\bG/B_k$.
The following lemma is almost immediate from this Panin's result.
\begin{lemma}  Let $\bF$ be a (twisted) flag variety.
Then the restriction map 
$res_{K}: K^*(\bF)\to K^*(\bar \bF)$  is injective.
\end{lemma}
\begin{proof} 
Recall that $res_{CH}\otimes \bQ:CH^*(\bF)\otimes\bQ \to CH^*(\bar \bF)\otimes \bQ$ is isomorphic.
Hence by the AHss, we see $K^*(\bF)\otimes \bQ\cong 
K^*(\bar \bF)\otimes \bQ$.  The following diagram
\[ \begin{CD}
K^*(\bF)    @>>>    K^*(\bar \bF)\\
@V{inj}VV       @V{inj}VV \\
K^*(\bF)\otimes \bQ @>{\cong}>> K^*(\bar \bF)\otimes \bQ.
\end{CD}\]
implies the lemma. \end{proof}

%Hence we have maps of 
% graded rings
% \[  gr_{\gamma}^*(\bF)\to gr_{geo}^*(\bF)
%\to gr_{geo}(\bar \bF)= gr_{top}^*(G/T).\]

For simply connected Lie group $G$, it is known
$res_K$ is surjective from Chevalley.  
\begin{thm} (Chevalley, Panin)
When $G$ is simply connected, $res_K$ is 
an isomorphism.
\end{thm}

Hence
we have Theorem 1.2 in the introduction.  However we will see it directly
and explicitly for each simple Lie group.
Moreover we will compute explicitly $gr_{\gamma}(X)$
for each simply connected Lie group.

\begin{lemma} Let $\bF$ be a versal complete flag variety.
The restriction map $res_{K}$ is isomorphic if and only if
for each generator $y_i\in P(y)$ in $CH^*(\bar \bF)$,
there is $c(i)\in K^*(BT)$ such that
$             c(i)=v_1^{s_i}y_i$, for $ s_i\ge 0$ in  
$ K^*(\bar \bF)/p.$
\end{lemma}
\begin{proof}
Consider the restriction map
\[res_{k}: k^*(R(\bG))\to k^*(\bar R(\bG))\cong k^*\otimes
P(y)\]
where $k^*=\bZ_{(p)}[v_1]$.  Suppose that $res_{K}=res_k[v_1^{-1}]$ is surjective.  Then since $\bF$ is
versal,
there is $c(i)\in K^*(BB_k)$ with $c(i)=v_1^{s_i}y_i$ for some $s_i\in\bZ$. Since $res_k(c(i))\subset k^*\otimes P(y)$,
we see $s_i\ge 0$.
The converse is immediate.
\end{proof}
\begin{cor} For each generator $y_i\in P(y)$,
if there is 
$x_{k(i)}\in \Lambda(x_1,...,x_{\ell})\subset H^*(G;\bZ/p)$
 such that $Q_1x_{k(i)}=y_i$, then
$res_K$ for a versal flag $\bF$ is isomorphic.
\end{cor}
\begin{proof}
Let $d_r(x_k)=b_k\in H^*(BT)/p$.  Then by Corollary 2.2,  
we have
\[ b_k=v_1y_ i\ \  mod(v_1^2)\quad in\ k(1)^*(G/T).\]
By induction, we can take generators $y_i$ satisfy the above 
lemma.
\end{proof}

  \begin{cor}
Let $G=G_1\times G_2$, and $\bF=\bG/B_k$, $\bF_i=\bG_{i,k}/B_{i,k}$
for $i=1,2$.  Suppose $\bF,\bF_i$ are versal.
If $res_K$ are isomorphic for $K^*(\bF_i)$ then so is for
$K^*(\bF)$.
\end{cor} 
\begin{proof}  Since $\bF$ is versal, we see
$  CH^*(\bF_1)\otimes CH^*(\bF_2)\to CH^*(\bF)$
is surjective,  because  $CH^*(\bF)$ generated by
\[  CH^*(B(B_{1,k}\times B_{2,k}))\cong CH^*(BB_{1,k}))\otimes CH^*(BB_{2,k}).\]
Let $P_i(y)\subset H^*(G_i/T_i)$ be the polynomial parts
corresponding to that in  $H^*(G_i;\bZ/p)$.
Then $P(y)\cong P_1(y)\otimes P_2(y)$.
By the assumption, $res_K$ are isomorphic for $\bF_i$.
Hence from the above lemma,
for $ y_i\in P_i(y)$, we have $v_1^{s_i}y_i=c(i)\in \Omega^*(BB_{i})$.
So  for $y_1y_2\in P(y)$, we see
\[ y_1y_2=v_1^{s_1+s_2}c(1)c(2)\quad with \ c(1)c(2)\in \Omega^*(BB_k)\cong BP^*(BT).\]
\end{proof}

\begin{proof}[The proof of Theorem 1.2 without
using Chevalley's theorem.]
Each simply connect compact Lie group is a product of
simple Lie groups. For each simple group except for
$E_7,E_8$ with $p=2$, we will show
the existence of $x_k$ with $Q_1x_k=y_i$ for each 
generator $y_i\in P(y)$ in $\S 7- \S 9$.
(For example $Q_1(x_1)=y_1$ for all simply connected
simple Lie groups.)
For groups $E_7,E_8$ and $p=2$,  we will see that 
we can take $b_i=v_1^sy_i$ $mod(2)$, for $s=1$ or $2$
in $\S 10, \S 11$.
\end{proof}

Here we consider the following (modified) AHss (for topological spaces)
\[ E_2^{*,*'}\cong H^*(BT;K^{*'}(G))\Longrightarrow
  K^*(G/T).\]
\begin{cor} For the above spectral sequence
$E_r^{*,*'}$, we have the isomorphism
\[ E_{\infty}^{*,*'}\cong K^{*'}\otimes gr_{\gamma}^{*}(G/T)
\quad (e.g.,\  E_{\infty}^{*,0}\cong gr_{\gamma}^*(G/T)). \]
\end{cor} 
\begin{proof}  When $\bF$ is versal,  $CH^*(\bF)$ is generated by elements in $CH^*(BT)$.
Since $K^*(\bF)\cong K^*(G/T)$, from Theorem 1.1,  $K^*(G/T)$is generated by
$K^*(BT)\cong K^*\otimes S(t)$.  Hence 
\[E_{\infty}^{*,*'}\cong K^{*'}\otimes S(t)/J \quad for\ some\ J.\]
By Atiyah (p. 63 in \cite{At}), if there is a filtration and the associated spectral sequence (which may not be AHss defined from skeletons) such that the $E_{\infty}$ is generated by Chern classes, then it is isomorphic to  $K^*\otimes gr_{\gamma}(G/T)$. 
\end{proof}

However, the study of the above spectral sequence seems
not so easy.  We use the following lemma,
to seek $gr_{\gamma}(G/T)$.
\begin{lemma}  Suppose that there are  a filtration on $K^*(R(G))$ and a $\bZ$-module $B$
with    $B/p\subset CH^*(R(\bG))/p$
such that
\[(1)\quad grK^*(R(\bG))\cong K^*\otimes B.\]
Moreover suppose that for any torsion  element $b\in B$ with $b\not =0\in B/p$, if there  exists $r,s>0$ such that in $K^*(R(\bG))$
\[ (2)\quad p^rb=v_1^sb'\quad for \ some\ b'\in CH^*(R(\bG_k)),\]
then $b'\in B$.
With these supposition, we have  $B\cong gr_{\gamma}(R(\bG))$.
\end{lemma}
\begin{proof}
Recall $gr_{\gamma}(R(\bG))\cong
CH^*(R(\bG))/I$ for some ideal $I$.
Since $B\subset grK^*(R(\bG))$, we see $B\cap I=0$.

Suppose that $b'\in gr_{\gamma}(R(\bG))$ but $b'\not \in B$.
Since $B$ generates $K^*(R(\bG))$ as a $K^*$-module,
there is $\tilde b\in B$ such that $b'=v_1^{-s}\tilde b$ for $s>0$ (otherwise $b'\in B$)
in $K^*(R(\bG))$ (but not in $grK^*(R(\bG))$). 
Let us write $\tilde b=p^rb$ in $K^*(R(\bG))$.
It contradicts to $(2)$.
Note that  in (2), $r>0$, otherwise $b=0\in CH^*(R(\bG))$.

  Since $p^rb=0\in CH^*(R(\bG))$,
 we only consider  a torsion element for $b$.
\end{proof}

%\begin{cor}
%Suppose that  there is $B$ satisfying (1)
%and all torsion elements are just $p$-torsion.
%Then $gr_{\gamma}(R(\bG))\cong B$.
%\end{cor}
To seek $gr_{\gamma}(R(\bG))$, we first find
some graded algebra $gr(K^*(R(\bG))\cong K^*\otimes
B$ such that each element $0\not =\tilde b\in B$ is represented
by an element (possible zero) in $CH^*(R(\bG))$.  Since $\tilde b\not =0\in 
K^*(R(\bG))$, there is $r,s\ge 0$ such that 
\[ p^r\tilde b=v_1^sb', \quad with \ b'\not= 0\in CH^*(R(\bG))
\cong k^*(R(\bG))\otimes _{k^*}\bZ_{(p)}.\]
Replacing $p^r\tilde b$ by $b'$ in $B$, we may get $gr_{\gamma}(R(\bG))$.

\section{The orthogonal group $SO(m)$ and $p=2$}

 We  consider the
orthogonal groups $G=SO(m)$ and $p=2$
in this section.
The mod $2$-cohomology is written as ( see for example \cite{Mi-Tod}, \cite{Ni})
\[ grH^*(SO(m);\bZ/2)\cong \Lambda(x_1,x_2,...,x_{m-1}) \]
where $|x_i|=i$, and the multiplications are given by $x_s^2=x_{2s}$. (Note that  the suffix means its degree.)

For ease of argument,  we only consider in the case
$m=2\ell+1$ (the case $m=2\ell+2$ works similarly) so that
\[ H^*(G;\bZ/2)\cong P(y)\otimes \Lambda(x_1,x_3,...,x_{2\ell-1}) \]
\[ grP(y)/2\cong \Lambda(y_2,...,y_{2\ell}), \quad 
letting\ y_{2i}=x_{2i}\ \ (hence \ y_{4i}=y_{2i}^2).\]

The Steenrod operation is given as 
$Sq^k(x_i)= {i\choose k}(x_{i+k}).$
The $Q_i$-operations are given by Nishimoto \cite{Ni}
\[Q_nx_{2i-1}=y_{2i-2^{n+1}-2},\qquad Q_ny_{2i}=0.\]

It is well known that   the transgression 
$d_{2i}(x_{2i-1})=c_i$ is the  $i$-th elementary symmetric function
on $S(t)$. 
 Moreover we see that 
$Q_0(x_{2i-1})=y_{2i}$ in $H^*(G;\bZ/2)$.
In fact, the cohomology  $H^*(G/T)$ is computed 
completely by
Toda-Watanabe \cite{Tod-Wa}
\begin{thm} (\cite{Tod-Wa}) 
There are $y_{2i}\in H^*(G/T)$ for $1\le i\le \ell$
such that $\pi^*(y_{2i})=y_{2i}$ for $\pi: G\to G/T$, and that 
we  have an isomorphism
\[ H^*(G/T)\cong \bZ[t_i,y_{2i}]/(c_i-2y_{2i},J_{2i})\]
where $J_{2i}= 1/4(\sum_{j=0}^{2i}(-1)^jc_jc_{2i-j})$
%=y_{4i}-\sum_{0<j<2i}(-1)^jy_{2j}y_{4i-2j}$ \\
letting $y_{2j}=0$ for $j>\ell$.
\end{thm}
By using Nishimoto's result for $Q_i$-operation, 
from Corollary 2.2, we have
\begin{cor} In $BP^*(G/T)/\II$, we have  
\[c_i= 2y_{2i}+\sum v_n(y(2i+2^{n+1}-2)) \]
for some $y(j)$ with $\pi^*(y(i))=y_{i}$.
\end{cor}

We have $c_i^2=0$ in $CH^*(\bF)/2$ from
the natural inclusion $SO(2\ell+1)\to Sp(2\ell+1)$
(see \cite{Pe}, \cite{Ya6}) for the symplectic group
$Sp(2\ell+1)$.
Thus we have 
\begin{thm} (\cite{Pe}, \cite{Ya6}) 
Let $(G,p)=(SO(2\ell+1),2)$ and $\bF=\bG/B_k$
be versal.  
Then $CH^*(\bF)$ is torsion free, and 
\[ CH^*(\bF)/2\cong S(t)/(2,c_1^2,...,c_{\ell}^2),\quad
CH^*(R(\bG))/2\cong \Lambda(c_1,...,c_{\ell}).\]
\end{thm}
\begin{cor} We have 
$ gr_{\gamma}(\bF)\cong gr_{geo}(\bF)\cong CH^*(\bF).$
\end{cor} 
\begin{proof} It is known  that the image $Im(d_r)$ of the differentials
of AHss are generated by
torsion elements.  Hence the ideal $I=\cup_rIm(d_r)=0$.
From Lemma 4.2, we see $gr_{geo}(\bF)\cong CH^*(\bF)/I\cong CH^*(\bF).$ 
\end{proof}

From Corollary 2.2,, we see
$ c_i=2y_i+v_1 y_{2+2}$ $mod (v_1^2).$
Here we consider
the $mod(2)$ theory version
$ res_{K/2}: K^*(\bF)/2 \to  K^*(\bar \bF)/2.$
\begin{lemma}
We have  $ Im(res_{K/2})\cong K^*\otimes \Lambda(y_{2i}|2i\ge 4).$  Hence $res_K$ is not surjective.
\end{lemma}
\begin{proof}
In $mod(2)$, we have $c_i=v_1y_{2i+2}$.  Hence for all
$i\ge 2$, we have  $y_{2i+2}\in Im(res_{K/2})$. 
Suppose $y_2y\in Im(res_{K/2})$ for $0\not =y\in \Lambda(y_{2i}|i\ge 2)$.  Then there is  $c\in \Lambda(c_i)$
such that $c=v^sy_2y$.
Since $c_i=v_1y_{2i+2}$, 
we can write $c=v^ty'$ for $y'\in \Lambda(y_{2i}|i\ge 2)$.
Hence $v^t(y'-v_1^{s-t}y_2y)=0$ in $K^*(\bar \bF)/2=K^*/2\otimes
\Lambda(y_{2i})$, which is $K^*/2$-free.
and  this is a contradiction.
\end{proof}
In the next section, we will compute the case 
$(G',p)=(Spin(2\ell+1),2)$.
It is well known that
$ G/T\cong G'/T'$ for the maximal torus $T'$ of the spin group.
Hence we see
\begin{lemma}  Let $\bF=\bG/B_k$ and $\bF'=\bG'/B_k'$
are versal.  Then
\[ gr_{\gamma}(\bF)\not \cong gr_{\gamma}(\bar \bF)\cong 
 gr_{\gamma}(\bar \bF')\cong gr_{\gamma}(\bF').\]
\end{lemma}
\begin{proof} The last two isomorphisms follow from
the facts that $\bF'$ is versal and  $res_K$ is isomorphic for spin groups
(see Lemma 6.3 below).
\end{proof}
Note that $c_2c_3\not =0$ in $gr_{\gamma}(\bF)$ but
it is zero in $gr_{\gamma}^*(\bF')$ from Theorem 7.8
for $Spin(n)$ when  $n\le 11$.
Hence the   restriction map
$ res_{gr_{\gamma}}:gr^*_{\gamma}(\bF)\to gr_{\gamma}(\bar \bF)$
is not injective,  while $res_K$ is injective.

\section{The spin group $Spin(2\ell+1)$ and $p=2$}

Throughout this section, let $p=2$,  $G=SO(2\ell+1)$
and $G'=Spin(2\ell+1)$.  By definition, we have the 
$2$ covering $\pi:G'\to G$.
It is well known that 
$\pi^*:   H^*(G/T)\cong H^*(G'/T')$.
  
Let $2^t\le \ell < 2^{t+1}$, i.e. $t=[log_2\ell]$.
The mod $2$ cohomology is
\[ H^*(G';\bZ/2)\cong H^*(G;\bZ/2)/(x_1,y_1)\otimes
\Lambda(z) \]
\[ \cong P(y)'\otimes \Lambda(x_3,x_5,...,x_{2\ell-1})\otimes \Lambda(z),\quad |z|=2^{t+2}-1\]
where
$P(y)\cong \bZ/2[y_2]/(y_2^{2^{t+1}})\otimes P(y)'$.  
(Here $d_{2^{t+2}}(z)=y^{2^{t+1}}$ for $0\not =y\in H^2(B\bZ/2;\bZ/2)$ in the spectral sequence induced from
the fibering $G'\to G\to B\bZ/2$.)
Hence
\[ grP(y)'\cong \otimes _{2i\not =2^j}\Lambda(y_{2i})\cong
\Lambda(y_6,y_{10},y_{12},...,y_{2\ell}).\]
The $Q_i$ operation for $z$ is given by Nishimoto 
\cite{Ni}
\[ Q_0(z)=\sum _{i+j=2^{t+1},i<j}y_{2i}y_{2j}, \quad 
 Q_n(z)=\sum _{i+j=2^{t+1}+2^{n+1}-2,i<j}y_{2i}y_{2j}\ \ for\ n\ge 1.\]

We know that 
\[ grH^*(G/T)/2\cong P(y)'\otimes \bZ[y_2]/(y_2^{2^{t+1}})\otimes S(t)/(2,c_1,c_2,...,c_{\ell}) \]
\[ grH^*(G'/T')/2\cong P(y)'\otimes S(t')/(2,c_2',.....,c_{\ell}',c_1^{2^{t+1}}).\]
Here $c_i'=\pi^*(c_i)$ and $d_{2^{t+2}}(z)=c_1^{2^{t+1}}$ in the spectral sequence
converging $H^*(G'/T')$.)
These are isomorphic, in  particular, we have
\begin{lemma}  The element $\pi^*(y_2)=c_1\in S(t')$
and $\pi^*(t_j)=c_1+t_j$ for $1\le j\le \ell$.
\end{lemma}

Take $k$ such that $\bG$ is a versal $G_k$-torsor so that
$\bG'_k$ is also a versal $G_k'$-torsor.  Let us write
$\bF=\bG/B_k$ and $\bF'=\bG'/B_k'$.  Then
\[ CH^*(\bar R(\bG'))/2\cong P(y)'/2,\quad and \quad
    CH^*(\bar R(\bG))/2\cong P(y)/2.\]
The Chow ring $CH^*(R(\bG'))/2$ is not computed yet
(for general $\ell$),
while we have the following lemmas.
\begin{lemma}  
We have a surjection 
\[  \Lambda(c_2',...,c_{\ell}', c_1^{2^{t+1}})\stackrel{ }{\twoheadrightarrow}
CH^*(R(\bG'))/2.\] 
\end{lemma}

\begin{lemma}  The restriction $res_K:
K^*(\bF')\to K^*(\bar \bF') $ 
 is isomorphic.
\end{lemma}
\begin{proof}    Recall the relation
\[ c_i'=2y_{2i}+v_1y_{2i+2i}\quad mod(v_1^2)\ \ in \ K^*(\bar \bF').\]
We consider the $mod(2)$ restriction map $res_{K/2}$.
Since $c_i'=v_1y_{2i+2}$ $mod(2)$, we see $y_{2i+2}\in Im(res_{K/2})$.
So $y_{2i'}\in Im(res_{K/2})$ for $2i'\ge 6$, e.g,,
for all $y_{2i} \in P(y)'=\Lambda(y_{2i}|i\not =2^j)$.  Hence $res_{K/2}$ is surjective.
Thus  $res_K$ itself surjective.
\end{proof}
\begin{cor}   Let $G''=Spin(2\ell+2)$ and $\bF''=\bG''/B_k$.
Then $res_K: K^*(\bF'')\to K^*(\bar \bF'')$ is isomorphic.
\end{cor}
\begin{proof}
This is immediate from $P(y)'\cong P(y)''$.
\end{proof}

\begin{cor}
We have $K^*(R(\bG'))/2\cong 
K/2^*\otimes  \Lambda(c_i'|i\not =2^j-1).$
\end{cor}
\begin{proof}
Recall $c_{2^j-1}'=2(y_{2^j-1})+v_1y_{2^j}$
where $y_{2^j}=c_1^{2^j}=0 \in K^*(\bar R(\bG))/2$.
So $c_{2^j-1}'=0\in K^*(R(\bG))/2$, since $res_K$ is 
injective.
We have the maps 
\[  K^*/2\otimes\Lambda(c_{i-1}'|i\not =2^j)
\twoheadrightarrow K^*(R(\bG))/2\]
\[ \stackrel{isom.}{\to}
K^*(\bar R(\bG))/2\cong 
 K^*/2\otimes \Lambda(y_{2i}|i\not =2^j)
\]
by $c_{i-1}'\to v_1y_i$.  This map is isomorphic from the above lemma.  Hence we have the corollary.
\end{proof}

Hence  $c_{2^j-1}'$ is not a module generator in 
$K^*(R(\bG))$.  So we can write
\[ 2^rb=v_1^sc_{2^j-1}'\quad r,s\ge 1,\quad \]
so that $b$ is torsion element in $CH^*(R(\bG))$.
In fact we have
\begin{lemma}  Let $k<t$ i.e., $2^{k+1}<\ell$.  Then in
$K^*( R(\bG))$, we have
\[ 2^{2^k}c_{2^k}'+\sum (-1)^i2^{2^k-i}v_1^ic_{2^k+i}'=v_1^{2^k}c_{2^{k+1}-1}'.\]
\end{lemma}
\begin{proof}
Let us write 
$c_{2^k}''=c_{2^k}'-2y_{2(2^k)}=v_1y_{2(2^k+1)}.$
Then 
\[ v_1y_{2(2^k+2)}=c_{2^k+1}'-2y_{2(2^k+1)}
=c_{2^k+1}'-2v_1^{-1}c_{2k}''.\]
By induction on $i$, we have
\[ v_1y_{2(2^k+i+1)}= (-1)^i2^iv_1^{-i}c_{2^k}''+\sum_{0<j<i}(-1)^j2^jv_1^{-j}c_{2(2^k+j)}'.\]

When $i=2^k-1$, we see $y_{2(2^{k+1})}$ is in 
$K^*(\bF)$
(it is zero in $K^*(R(\bG))$ from Lemma 3.6).
Hence we have the equation in the lemma.
\end{proof}

In general,  $gr_{\gamma}(R(\bG))$ seems complicated.  Here we only note the following  proposition. 
Let us write $\bZ_{(p)}$-free module
\[ \Lambda_{\bZ}(a_1,...,a_n)=\bZ_{(p)}\{a_{i_1}...a_{i_s}|
1\le i_1<...<i_s\le n\}\]

\begin{prop}
There is an additive  split surjection 
\[gr_{\gamma}(R(\bG))/2\twoheadrightarrow  B_1/2+B_2/2\]
for degree $*\le |c_1^{2^{t+1}}|=2^{t+2}$ where
\[B_1=\Lambda_{\bZ}(c_i'|i\not=2^j-1,\ 2\le i\le \ell-1),\quad
B_2=\Lambda_{\bZ}(c_i'| i \not = 2^j,\ 3\le i\le \ell),\]
so that $B_1\cap B_2\cong \lambda_{\bZ}(c_i'|i\not =2^j,2^j-1,
\ 5\le i\le \ell-1)$.
\end{prop}
\begin{proof}
First note that when $*<2^{t+1}$,
we have the surjection  $\Lambda(c_2',...,c_{\ell})'
\twoheadrightarrow CH^*(R(\bG'))/2$.

Recall $k^*(X)$ is the $K$-theory with $k^*(pt)\cong
\bZ_{(p)}[v_1]$. 
 We consider the map
\[ \rho:k^*/2\otimes B_1\to k^*(R(\bG))/2\]
\[ \to k^*(\bar R(\bG))/2\subset K^*/2\otimes \Lambda(y_{2i}|
i\not =2^j)\cong  K^*(\bar R(\bG))/2.\]
We show  
\[ Im(\rho)=k^*/2\{v_1^sy_{2i_1}...y_{2i_s}|i_k\not =2^j,
 1\le i_1<...<i_s\le \ell-1\}.\] 
In fact, $Im(\rho)$ contains the right  hand side module 
since $c_i'\mapsto v_1y_{2i+2}$. Suppose  that 
$\rho(b_I)=v_1^ky_{2i_1}...y_{2i_s}$ for $b_I\in B_1$.
Then $b_I$ contains $c_{i_1-1}...c_{i_s-1}$, which must 
maps to $v_1y_{i_1}...v_1y_{i_s}$.  Hence $k\ge s$.

Therefore each generator $c_I=c_{i_1}...c_{i_s}\in \Lambda(c_i|
i\not= 2^j-1)$ is also $k^*$-module generator of
$Im(\rho)$.  Hence it is nonzero in $CH^*(R(\bG))/2$
from $k^*(X)\otimes _{k^*}\bZ_{(p)}\cong CH^*(X)$.
Thus $B_1/2\subset CH^*(R(\bG))/2$.

Net we consider for $B_2$.
 We consider the map
\[ \rho_{\bQ}:k^*\otimes B_2\to k^*(R(\bG))\to k^*(\bar R(\bG))\]
\[ \twoheadrightarrow  k^*(\bar R(\bG))\otimes_{k^*}
 \bQ\cong CH^*(\bar R(\bG))\otimes \bQ\cong 
 \Lambda_{\bZ}(y_{2i}|i\not =2^j)\otimes \bQ.\]
by $c_i\mapsto 2y_{2i}$.
By the  arguments similar to the case $B_1$,  we see 
\[ Im(\rho_{\bQ})=k^*\{2^sy_{2i_1}...y_{2i_s}|i_k\not =2^j,
 1\le i_1<...<i_s\le \ell\}.\]
 Therefore each generator $c_I=c_{i_1}...c_{i_s}\in \Lambda(c_i|
i\not= 2^j)$ is also $k^*$-module generator of
$Im(\rho_{\bQ})$.  Hence $B_2/2\subset CH^*(R(\bG))/2$.
\end{proof} 
 
{\bf Remark.} Note that $c_{2^k}c_{2^j-1}$ contains $2v_1y_{2^{k+1}+2}y_{2^{j+1}-2}$, while this element
is also contained in $c_{2^k+1}c_{2^j-2}$.

Now we consider examples. 
For groups $Spin(7),$ $Spin(9)$,
the graded ring $gr_{\gamma}(G')/2$
are given in the next section (in fact these groups are of type $(I)$).
We consider here the group $G'=Spin(11)$.
The cohomology is written as 
\[H^*(G';\bZ/2)\cong \bZ/2[y_6,y_{10}]/(y_6^2,y_{10}^2)\otimes \Lambda(x_3,x_5,x_7,x_9,z_{15}).\]
By Nishimoto, we know $Q_0(z_{15})=y_6y_{10}$. It implies
$2y_6y_{10}=d_{16}(z_{15})=c_1^8$.
Since $y_{top}'=y_6y_{10}$, we have $t(G')=2$.
\begin{thm}
 For $(G',p)=(Spin(11),2)$, we have the isomorphisms
\[gr_{\gamma}(R(\bG))/2\cong \bZ/2\{1, c_2',c_3',c_4',c_5',c_2'c_4', c_1^8\},\]
\[ gr_{\gamma}(\bF)/2\cong S(t)/(2,c_i'c_j', c_i'c_1^8,c_1^{16}|2\le i\le j\le 5,\ (i,j)\not
=(2,4)).\]
\end{thm}
\begin{proof}
Consider the restriction map $res_K$
\[ K^*(R(\bG))\cong  
\tilde K^*\{1,c_2,c_4,c_2c_4\}
 \stackrel{\cong }{\to} K^*\{1,y_6,y_{10},y_6y_{10}\}
\cong K^*(\bar R(\bG))\]
by $res_K(c_2)= v_1y_6$, $res_K(c_4)=v_1y_{10}$.
Note $c_2c_4\not =0 \in CH^*(R(\bG))/2$,
from the preceding proposition. (In fact,
$v_1y_6y_{10}\not \in Im(res_K)$.)

Next using $res_K(c_3')=2y_6$, $res_K(c_5')=2y_{10}$,
and $res_K(c_1^8)=2y_6y_{10}$,  we have
\[grK^*(R(\bG))\cong
K^*/2\{c_2',c_4',c_2'c_4'\}\oplus
   K^*\{1,c_3',c_5',c_1^8\}.\]
From this and Lemma 5.7, we show the first isomorphism.
(Note generators $c_2',...,c_1^8$ are all nonzero
in $CH^*(R(\bG))/2$ by Corollary 3.8.)

Let us write $y_{6}^2=y_6t\in CH^*(\bar \bF)/2$ for $t\in S(t)$.  Then
\[ c_2'c_3'=2v_1y_6^2=2v_1y_6t=v_1c_3't\quad in \ K^*(\bF).\]
Hence we see $c_2'c_3'=0$ in $CH^*(\bF)/I\cong 
gr_{\gamma}^*(\bF)/2$.  
We also note
\[ c_3'c_5'=4y_6y_{10}=2c_1^8 \quad in\ K^*(\bar \bF).\]
By  arguments similar to the proof of Lemma 3.5, we can show the second isomorphism.
\end{proof}

\section{ Groups of type $(I)$ (e.g., $(G,p)=(E_8,5)$).}

When $G$ is simply connected and $P(y)$ is generated by just one generator, 
we say that $G$ is of type $(I)$.  Except for 
$(E_7,p=2)$ and $(E_8,p=2,3)$, all exceptional (simple) Lie groups are of type $(I)$.  Note that in these cases, it is known
$rank(G)=\ell\ge 2p-2$.

{\bf Example.}  The case $(G,p)=(E_8,5)$ is of type $I$.
we have \cite{Mi-Tod} 
\[H^*(E_8;\bZ/5)
\cong  \bZ/5[y_{12}]/(y_{12}^5)\otimes 
\Lambda(z_{3},z_{11},z_{15},z_{23},z_{27},z_{35},z_{39},z_{47})\]
where suffix means its degree.  The cohomology  operations are given 
\[ \beta(z_{11})=y_{12},\ \ \beta(z_{23})=y_{12}^2, \ \ 
\beta(z_{35})=y_{12}^3,\ \ \beta(z_{47})=y_{12}^4, \ \ \]
\[P^1z_{3}=z_{11} ,\ \ P^1z_{15}=z_{23},\ \ P^1z_{27}=z_{35},\ \ 
 P^1z_{39}=z_{47}.\]

Similarly, for each group $(G,p)$ of type $ (I)$,  we can write
\[H^*(G;\bZ/p)\cong \bZ/p[y]/(y^p)\otimes \Lambda(x_1,...,x_{\ell}) \quad |y|=2(p+1).\]
The cohomology operations are given as
\[ \beta: x_{2i}\mapsto y^i,\quad P^1:x_{2i-1}\mapsto 
x_{ 2i} \qquad for\ 1\le i\le p-1.\]
Hence we   have \ $ Q_1(x_{2i-1})=Q_0(x_{2i})=y^i
$ \ for $1\le i\le p-1$.

 \begin{thm}  
Let $G$ be of type $(I)$.  Then
$res_K$ is isomorphic, and   
\[gr^*_{\gamma}(R(\bG))  \cong
(\bZ/p\{b_1,b_3,...,b_{2p-3}\}\oplus \bZ_{(p)}\{1,b_2,...,b_{2p-2}\}).\]
Moreover $gr_{\gamma}(G/T)/p\cong gr_{\gamma}(\bF)/p$ is isomorphic to 
\[ S(t)/(p,b_ib_j,b_k|1\le i,j\le2p-2<k\le \ell).\]
\end{thm}
\begin{proof}
From $Q_1x_1=y$, we see $b_1=v_1y\ mod(v_1^2)$
in $k(1)^*(G/T)$.
From Corollary 5.4, we see $res_K$ is isomorphic, namely
the map
\[ K^*(R\bG_k)\cong K^*[b_1]/(b_1^p)\to 
 K^*[y]/(y^p)\cong  K^*(\bar R\bG)\]
is isomorphic by $res_k(b_1)=v_1y$. 

Since $Q_0x_2=y$ in $H^*(G;\bZ/p)$, we see $b_2=py$ in 
$ k^*(G/T)$.
Hence 
\[ pb_1=pv_1y=v_1b_2\quad  in\  k^*(G/T).\]
Since $Q_1x_3=y^2$, we see $b_3=v_1y^2$.  Hence
$ b_1^2=v_1^2y^2=v_1b_3.$
By induction on $i$, we have 
\[ b_1^i=v_1^{i-1}b_{2i-1},\quad  2b_1^i=v_1^{i-1}b_{2i}.\]
Thus we get the isomorphism
\[ gr K^*(R(\bG))=K^*(R(\bG))/p\oplus pK^*(R(\bG))\cong 
 K^*\otimes B\]
\[ where \quad
 B= (\bZ/p\{b_1,b_3,...,b_{2p-3}\}
\oplus \bZ_{(p)}\{ 1, b_2,b_4,...,b_{2p-2}\}).\]
Note that  $b_i$ are  nonzero in  $CH^*(R(\bG))/p$,
 it gives  $gr_{\gamma}(R(\bG))\cong B$
from Lemma 5.7.

The above $B/p$ is rewritten using $A=\bZ/p[b_1,...,b_{\ell}]$
\[ B/p\cong A/(b_ib_j,b_k| 1\le  i,j\le 2p-2<k).\]
From arguments similar to the proof of  Lemma 3.5,  we have the theorem.
\end{proof}

Here we recall the (original) Rost motive $R_a$
(we write it by $R_n$)
defined from a nonzero pure symbol $a$ in $K_{n+1}^M(k)/p$
(\cite{Ro1}, \cite{Ro2}, \cite{Vo2}, \cite{Vo3}).
We have  isomorphisms (with $|y|=2(p^n-1)/(p-1)$)
\[ CH^*(\bar R_n)\cong \bZ[y]/(y^p),\quad \Omega^*(\bar
R_n)\cong BP^*[y]/(y^p). \]
\begin{thm} (\cite{Vi-Ya}, \cite{Ya4}, \cite{Me-Su})
The restriction 
$ res_{\Omega}:\Omega^*(R_n)\to \Omega^*(\bar R_n)$
is injective.  Recall $I_n=(p,...,v_{n-1})\subset BP^*$. Then
\[Im(res_{\Omega})\cong  BP^*\{1\}\oplus I_n[y]^{+}/(y^p)
\subset BP^*[y]/(y^p).\]
Hence writing $v_jy^i=c_{j}(y^i)$, we have 
\[CH^*(R_n)/p\cong \bZ/p\{1,c_j(y^i)\ |\ 0\le j\le n-1,\ 
1\le i\le p-1\}.\]
%Hence writing $c_j(y^i)=v_jy^i$, we have
%\[ CH^*(R_n)/p\cong \bZ/p\{1,c_0(y),c_1(y),...,c_{n-2}(y^{p-1}), %c_{n-1}(y^{p-1})\}.\]
\end{thm}
{\bf Example.}  In particular,   we have  isomorphisms
%\[ CH^*(R(\bG))/p\cong \bZ/p\{1,c_{0}(y),c_{1}(y),...,c_{0}(y^{p-1}), c_{1}(y^{p-1})\}.\]
%\[ CH^*(R_1)/p\cong \bZ/p\{1,c_0(y),...,c_0(y^{p-1})\},\]
\[ CH^*(R_2)/p\cong \bZ/p\{1,c_0(y),c_1(y),...,c_0(y^{p-1}),c_1(y^{p-1})
\}.\]

By \cite{Pe-Se-Za}, it is known when $G$ is of type $(I)$,
we see $J(\bG)=(1)$ and $R(\bG))\cong R_2$.
\begin{cor}  We have $ gr_{\gamma}(\bF)\cong CH^*(\bF)_{(p)}.$
\end{cor}
\begin{proof}
We have the additive isomorphism
\[ gr_{\gamma}(R(\bG))\cong B=\bZ/p\{b_1,...,b_{2p-3}\}\oplus
\bZ_{(p)}\{1,b_2,...,b_{2p-2}\}\]
\[ \cong \bZ/p\{c_1(y),...,c_1(y^{p-1})\}\oplus
\bZ_{(p)}\{1,c_0(y),...,c_0(y^{p-1})\},\]
which is isomorphic to $CH^*(R_2).$
The fact that  $gr_{\gamma}(\bF)\cong CH^*(\bF)/I$ for some ideal $I$ implies $I=0$ and the theorem.
\end{proof}
The above $I$ is nonzero for $R_n$ when $n\ge 3$.
\begin{lemma}
There is an isomorphism for  $n\ge 3$,  
  \[ gr^{2*}_{geo}(R_n)\cong CH^*(R_n)/I\quad with\
  I=\bZ/p\{c_2,...,c_{n-1}\}[y]/(y^{p-1}).\]
\end{lemma}
\begin{proof}   First, note  $v_j=0$ in $K^*=\bZ_{(p)}[v_1,v_1^{-1}]$ for $j\not =0,1$.  We have   
\[ v_1c_j(y)=v_1v_jy=v_jc_1(y)\quad in \ \Omega^*(\bar R_n).\]
Since $res_{\Omega}$ is injective, we see
$v_1c_j(y)=v_jc_1(y)=0$ in $K^*(R_n)/p$.
\end{proof}

\section{The case $G=E_8$ and $p=3$}

 Throughout this section, let
$(G,p)=(E_8,3)$.
  The cohomology $H^*(G;\bZ/3)$ is isomorphic to
([Mi-Tod]) 
\[ \bZ/3[y_{8},y_{20}]/(y_8^3,y_{20}^3)\otimes
\Lambda(z_3,z_7,z_{15},z_{19},z_{27},z_{35},z_{39},z_{47}).\]
Here the suffix means its degree, e.g., $|z_i|=i$.
By Kono-Mimura \cite{Ko-Mi} the actions of cohomology 
operations
are also known
\begin{thm} (\cite{Ko-Mi})
We have $P^3y_8=y_{20}$, and 
\[ \beta: z_7\mapsto y_8,\ \ z_{15}\mapsto y_8^2,\ \ 
z_{19}\mapsto y_{20},\ \
 z_{27}\mapsto y_{8}y_{20},\ \ z_{35}\mapsto y_{8}^2y_{20},\]
\[
z_{39}\mapsto y_{20}^2,\ \ z_{47}\mapsto y_8y_{20}^2,\] 
\[ P^1:\ z_3\mapsto z_7,\ \  z_{15}\mapsto z_{19},\ \  z_{35}
\mapsto z_{39}\] \[  P^3:\ z_7\mapsto z_{19},\ \  z_{15}\mapsto z_{27}
\mapsto -z_{39},\ \
 z_{35}\mapsto z_{47}.\]
\end{thm}
We use notations $y=y_8,y'=y_{20}$, and $x_1=z_3,...,x_8=z_{47}$.
Then we can rewrite the isomorphisms 
\[ H^*(G;\bZ/3)\cong \bZ/3[y,y']/(y^3,(y')^3)\otimes 
\Lambda(x_1,...,x_{8}), \]
\[ grH^*(G/T;\bZ/3)\cong \bZ/3[y,y']/(y^3,(y')^3)\otimes 
S(t)/( b_{1}, ,...,b_{8}).\]
From Corollary 2.2, we have 
\begin{cor} (\cite{Ya6})
We can take $b_1\in BP^*(BT)$ such that 
 \[v_1y+v_2y'=b_1\quad in\  BP^*(G/T)/\II.\]
\end{cor} 

From the preceding theorem, we know that
all $y^i(y')^j$ except for $(i,j)=(0,0)$ and $(2,2)$ are $\beta$-image.
Hence we have 
\begin{cor} 
For all nonzero monomials  $u\in P(y)^+/3$
except for $(yy')^2$,  it holds  $3u\in S(t)$.
In fact, in  $BP^*(G/T)/I_{\infty}^2$
\[ b_1=v_1y+v_2y',\ \ b_2=3y,\ \ b_3=3y^2+v_1y' \ \ b_4=3y',\]
\[ b_5=3yy',\ \ b_6=3y^2y'+v_1(y')^2,\ \ b_7=3(y')^2,\ \ b_8=3y(y')^2.\]
 \end{cor}

\begin{lemma}  
The restriction map $res_K$ is isomorphic.
\end{lemma}
\begin{proof}
From $Q_1x_1=y$ and $Q_1x_3=y'$, we see $b_1=v_1y,
b_3=3y^2+v_1y'\ mod(v_1^2)$
in $k^*(G/T)/3$.  From Corollary 5.4, we have the lemma.
\end{proof}

In $K^*(R(\bG))$, let  
$\bar b_3=b_3-3(v_1^{-1}b_1)^2$ so that we have $\bar b_3=v_1y'$.  We have the following graded algebra, while
  $\bar b_3\not \in CH^*(R(\bG))$.  
\begin{lemma}  
Let us write $gr_3(K^*(R(\bG))=K^*(R(\bG))/3\oplus 3K^*(R(\bG))$.  Then 
\[gr_3K^*(R(\bG))  \cong 
K^*\otimes ( P(\bar b)\oplus \bar B_2)\quad where \]
\[ \begin{cases}
 P(\bar b)=\bZ_{(3)}[b_1,\bar b_3]/(b_1^3,\bar b_3^3),   \\ 
 \bar B_2=\bZ_{(3)}\{3,b_2,b_4,b_5,b_6',b_7,b_8\}
\oplus \bZ_{(3)}\{b_1b_2,b_1b_8\},\quad with\ b_6'=v_1b_6-b_3^2.
\end{cases}\]
\end{lemma}
\begin{proof}
By definition of $\bar b_3$, we see that the restriction map 
\[ K^*(R\bG_k)\cong  K^*[b_1,\bar b_3]/(b_1^3,
\bar b_3^3)\to
 K^*[y,y']/(y^3,(y')^3)\cong K^*(\bar R\bG_k)\]
is isomorphic by $res_K(b_1)=v_1y$ and $res_K(\bar b_3)=v_1y'$.
Here we note that
\[ b_2=3y=3v_1^{-1}b_1,\quad b_4=3y'=3v_1^{-1}\bar b_3,\quad
b_5=3v_1^{-2}b_1\bar b_3\]
\[ b_6=3v_1^{-3}b_1^2\bar b_3+v_1^{-1}\bar b_3^2,\quad
b_7=3v_1^{-2}\bar b_3^2,\quad
\quad b_8=3v_1^{-3}b_1\bar b_3^2.\]
Moreover we have $3b_1^2\bar b_3^2=3v_1^4(yy')^2=v_1^3b_1b_8$.

Therefore we have the isomorphism
\[ K^*\otimes (\bZ\{b_2,b_4,b_5,b_7,b_8\}
\oplus \bZ\{b_6'=v_1b_6-b_3^2\}\oplus \bZ\{b_1b_8\})\]
\[\cong \tilde K^*\otimes (3\bZ\{b_1,\bar b_3,b_1\bar b_3,\bar b_3^2,b_1\bar b_3^2\}\oplus
 3\bZ\{b_1^2\bar b_3\}\oplus 3\bZ\{b_1^2\bar b_3^2\}).\]
For the element  $3b_1^2$, we note  $ 3b_1^2=3v_1^2y^2=v_1b_1b_2.$

Thus we have  the graded ring $ K^*(R(\bG))/3\oplus 3K^*(R(\bG))$. \end{proof}

We will consider a filtration of $gr_3(K^*(R(\bG))$ and its
associated ring
\[gr'(K^*(R(\bG)))=gr(gr_3(K^*(R(\bG))),\]
which is isomorphic to $K^*\otimes B$ for some
$B\subset CH^*(R(\bG))$.

Since $b_3=\bar b_3 \ mod(3)$, we see
$P(\bar b)/3\cong P(b)/3$.  So we can replace $\bar P(\bar b)/3$ by $P(b)/3$ in $gr_3(K^*(R(\bG))$.
Using $3b_1b_2=v_1b_2^2$ and  $3b_1b_8=v_1b_2b_8$, we can write
\[  gr(K^*\otimes \{b_1b_2,b_1b_8\})
\cong K^*/3\{b_1b_2,b_1b_8\}
\oplus K^*\{b_2^2,b_2b_8\}.\] 
Since $b_3^2+b_6'=v_1b_6$, we can replace $b_6'$ by $b_6$. ( Note $|b_6|=36$  in  $ gr'(K^*(R\bG))$ but
$|b_6'|=32$ in $  gr_3(K^*(R(\bG))$.)

\begin{lemma} 
We have the injection (of graded modules)
\[ \bZ/3\{b_2,b_4,b_5,b_6,b_7,b_8\}\oplus
\bZ/3\{b_1b_2,b_1b_8,b_2^2,b_2b_8\}
\subset CH^*(R(\bG))/3.\]
\end{lemma}
\begin{proof}
From Corollary 3.8, we see
\[ \bZ/3\{b_1,b_2,...,b_8\}\subset CH^*(R(\bG))/3.\]

We see that $t(E_8)_{(3)}=9$  from
$res_{k}(b_2b_{8})=9(yy')^2$ (\cite{To1}).  We also see
$res_k((b_1{b_8})=3v_1(yy')^2$.
The fact $t(E_8)\not =3$ implies that 
 $b_2b_8,b_1b_8$ are 
 nonzero in $CH^*(R(\bG))/3$.

For the element $b_2^2$, we consider the restriction
$res_k(b_2^2)=9y^2$.
Since $3y^2\not \in Im(res_k)$, we see   
 $b_2^2\not =0\in CH^*(R(\bG))/3$. Since $3b_1b_2=v_1b_2^2$, we also see
$b_1b_2\not =0$ in $CH^*(R(\bG))/3$.
\end{proof}

%\begin{lemma} In $K^*(R(\bG))$, we have 
%$3b_3=v_1b_4+b_2^2$.
%\end{lemma}
%\begin{proof}
%Note that $3b_3-3\bar b_3=3(3y^2)=b_2^2$.
%The equation $3\bar b_3=3(v_1y')=v_1b_4$ implies
%the lemma. 
%\end{proof}

% \begin{thm}  
%Let $(G,p)=(E_8,3)$.  Then we have ;
%\[ gr_{\gamma}(R(\bG))\cong
% (B_1'\oplus B_2)\quad where \]
%\[ \begin{cases}
%B_1'=(\bZ/3[b_1,b_3]/(b_1^3,b_3^2)\oplus \bZ/3[b_1]/(b_1^3)\{b_6\} \oplus 
%   \bZ/3\{b_1b_2,b_1b_8\}),  \\ 
%  B_2=\bZ\{1,b_2,b_4,b_5,b_6,b_7,b_8\}\oplus %\bZ\{b_2^2,b_2b_8\}.
%\end{cases}\]
%\end{thm}
Then  we get the another graded ring $gr'(K^*(R(\bG))$.
\begin{prop}  
  There is a filtration whose 
associated graded ring is 
\[gr'K^*(R(\bG))  \cong 
K^*\otimes (B_1\oplus B_2)\quad where \]
\[ \begin{cases}
B_1=(P(b)/(3)\oplus 
   \bZ/3\{b_1b_2,b_1b_8\}),\quad P(b)=\bZ_{(3)}[b_1,
 b_3]/(b_1^3, b_3^3),   \\ 
  B_2=\bZ_{(3)}\{b_2,b_4,b_5,b_6,b_7,b_8\}
\oplus \bZ_{(3)}\{b_2^2,b_2b_8\}.
\end{cases}\]
\end{prop}

Suppose $(B_1\oplus B_2)/3\subset CH^*(R(\bG))/3$.  Then we  would have 
$gr_{\gamma}(R(\bG))/3\cong (B_1\oplus B_2)/3$
from lemma 5.7.  However the above supposition
is not correct.
\begin{cor}
There is a submodule $A\subset P(b)/3$ such that
\[gr_{\gamma}(R(\bG))/3\cong
(B_1/A)\oplus A'\oplus B_2/3\]
where $s:A\cong A'$ as non-graded module (e.g.,
$|a|<|s(a)|$ for $a\in A$).
\end{cor}
\begin{proof}
For $0\not =a\in P(b)$, if $a=0$ in $CH^*(R(\bG))$,
 then there
$a'\in CH^*(R(\bG))$ such that  $a=v_1^ka'$ in $k^*(R(\bG))$ for $k>0$.
Let $s(a)=a'$.  Then we have the corollary
applying Lemma 9.8.
\end{proof}

To study $A.A'$, we need some arguments ans lemmas.
Let $h^*(-)$ be a $mod(p)$ cohomology theory (e.g. $H^*(-;\bZ/p)$,
$k(n)^*(-)$).
The product $G\times G\to G$ induces the map
\[ \mu: G\times G/T\to G/T.\]
Here note $h^*(G\times G/T)\cong h^*(G)\otimes_{h^*}h^*(G/T)$, since $h^*(G/T)$ is
$h^*$-free.  For $x\in h^*(G/T)$, we say
that $x$ is $primitive$ ([Mi-Ni], [Mi-Tod])  if
\[ \mu^*(x)=\pi^*(x)\otimes 1\ +\ 1\otimes x\quad where \
\pi:G\to G/T.\]

Of course $b\in BP^*(BT)$ are primitive but $by_i$ are not,
in general.
We can take $y_1$ as primitive (adding elements
if necessary) in $BP^*(G/T)$.

For ease of arguments, we write $d(x)=1/4|x|$, e.g.,
\[ d(v_1)=-1,\ d(b_1)=1,\ d(b_2)=2,\ d(b_3)=4,\ d(b_4)=5 \]
\[ d(b_5)=7,\ d(b_6)=9,\ d(b_7)=10,\ d(b_8)=12.\]

We easily see that $b_1,b_3,b_1^2,b_3^2,b_1b_3$ are 
nonzero in $CH^*(R(\bG))/3$.  Hence 
\[ A\subset \bZ/3\{b_1^2b_3,b_1b_3^2,b_1^2b_3^2\}. \]
\begin{lemma}
We see $b_1^2b_3\not \in A$.
\end{lemma}
\begin{proof}
In $k^*(\bar R(\bG))$, we have 
\[ b_1^2b_3=(v_1y)^2(3y^2+v_1y')=v_1^3(y^2y') \quad
mod(I_{\infty}^4),\]
by using facts that $y^3$ is primitive and hence $y^3\in I_{\infty}$.
Suppose $b_1^2b_3=v_1b$ in $k^*(R(\bG))$
for $b\in P(b)=\bZ[b_1,...,b_{\ell}]$.  Then 
\[ (*)\quad b=(v_1)^2(y^2y')\quad mod(I_{\infty}^3),\quad d(b)=7.\]

If $b=b_k$ $mod(b_ib_j)$,
then by dimensional reason, 
\[ b_k=b_5=3yy'+\lambda v_1^2(y^2y'),\quad \lambda\in \bZ.\]
This case, of course,  does not satisfy $(*)$.

Let  $b=\sum b(1)b(2)$ for $b(i)\in P(b)^+$.
Since $b(i)\in I_{\infty}$, then $b(j)=3\lambda
y(i)+v_1y(i')$ for $y(i')\not =0\in P(y)$.
Hence each $b(i)$ must be one of  $b_1,b_3,b_6$.
By dimensional reason (such as $d(b)=7$, $d(b_6)=9$),
these cases do not happen.
We see (more easily) that 
the cases $b_1^2b_3=v_1^sb$ do not happen for $s\ge 2$.
\end{proof}

\begin{lemma} We have $b_1b_3^2\in A$.
\end{lemma}
\begin{proof}  First note 
\[ b_1b_3^2=v_1y(3y^2+v_1y')^2=v_1^3y(y')^2
\quad mod(I_{\infty}^4)\]
in $k^*(\bar R(\bG))$, using the primitivity of $y^3$.
Similarly we have 
\[ b_1b_6=v_1y(3y^2y'+v_1(y')^2)=v_1^2y(y')^2\quad mod(I_{\infty}^3).\] 
Let us write $b=b_1b_3^2-v_1b_1b_6$.  Then 
$b\in k^*\otimes P(b)$ and  in $(I_{\infty}^4).$

If $b$ contains (as a sum of) $3v_1^{s(i,j)}y^i(y')^j$
(or $9v_1^sy^2(y')^2$) for  $s,s(i,j)\ge 0$, then take off
\[v_1^{s(i,j)}b_k=3v_1^{s(i,j)}y^i(y')^j\ mod(v_1^{s(i,j)+1}), \quad 
v_1^sb_2b_8=9v_1^sy^2(y')^2\ mod(v_1^{s+1}).\]
Continue this argument, we have
\[b'= b-\sum_{k}\lambda_k v_1^{s'}b_k %-\lambda'v_1^{r}b_1b_2
-\lambda''v_1^{r'}b_1b_8-\lambda'''
v_1^{r''}b_2b_8\]
\[=\sum \mu_{k'}v_1^{s''}y^i(y')^j
+3\mu v_1^sy^2(y')^2,\quad \mu\not =0\ mod(3).\]
This element is in $I_{\infty}^4$, and moreover,
by the dimensional reason such as 
$d(b')=9$ and $d(v_1^4y(y')^2)=8$,  it is represented as
$ b'=\mu' v_1^5y^2(y')^2.$

Next we see
\[ b'-\mu'v_1^2b_1^2b_6 =\mu' v_1^4(3y^3y')
\quad mod(v_1^6)
\quad (in \ k^*(\bar R(\bG)).\]
The ideal $ (9)\subset k^*(\bar R(\bG))$ is in
$Im(res_k)$.  Therefore we have seen 
$b\in I_{\infty}^2Im(res_k)$.
\end{proof}
\begin{lemma}  We have    
\[A= \bZ/3\{b_1b_3^2,b_1^2b_3^2\},
\quad and \quad A'=\bZ/3\{ b_1b_6,b_1^2b_6\}. \]
\end{lemma}
\begin{proof}
From the above lemma, we can take
$b_1b_2^2=v_1b_1b_6$ in $k^*(\bar R(\bG))$.
Hence $b_1^2b_2^2=v_1b_1^2b_6$.  So it is sufficient
to prove that $b_1^2b_6\not =0 \in CH^*(R(\bG))/3$.
Let $b_1^2b_6=vb$.  Then
\[ (*)\quad b=v_1^2y^2(y')^2\quad mod (I_{\infty}^3).\]

If $b=b_k$ $mod(b_ib_j)$,
then by dimensional reason, 
$ b_k=b_8=3y(y')^2\ mod (v_1)$ and it is a contradiction.
Let  $b=\sum b(1)b(2)$ for $b(i)\in P(b)^+$.
Hence each $b(i)$ must be one of  $b_1,b_3,b_6$.
By dimensional reason (such as $d(b_1)=1,d(b_3)=4, d(b_6)=9$ and $d(b)=12$),
these cases do not happen.
\end{proof}
\begin{thm}
Let $(G,p)=(E_8,3)$. For $B_1,B_2$ in Proposition 9.8,
\[ gr_{\gamma}(R(\bG))/3\cong
B_1/(3,b_1b_3^2,b_1^2b_3^2)\oplus
\bZ/3\{b_1b_6,b_1^2b_6\}\oplus B_2/3.\]
\end{thm}

\section{The case $G=E_7$ and $p=2$.}

Throughout this section, let
 $(G,p)=(E_7,2)$.
\begin{thm} 
The cohomology $grH^*(G;\bZ/2)$ is given 
\[ \bZ/2[y_6,y_{10},y_{18}]
/(y_6^2,y_{10}^2,y_{18}^2),
\otimes \Lambda(z_3,z_5,z_9,z_{15},z_{17},z_{23},z_{27}).\]
\end{thm}
Hence we can rewrite
\[ grH^*(G;\bZ/2)\cong \bZ/2[y_1,y_2,y_3]/(y_1^2,y_2^2,y_3^2)\otimes \Lambda(x_1,...,x_7).\]
\begin{lemma}  The cohomology operations acts as
\[\begin{CD}
  x_1=z_3 @>{Sq^2}>> x_2=z_5@>{Sq^4}>>
 x_3=z_9 @>{Sq^8}>> x_4=z_{17}\\
 x_5=z_{15} @>{Sq^8}>> x_6=z_{23}@>{Sq^4}>>
 x_7=z_{27} 
\\
x_5=z_{15}  @>{Sq^2}>> x_4=z_{17} @. @. @.
\end{CD} \]
The Bockstein acts $Sq^1(x_{i+1})=y_i$ for $1\le i\le 3$,
  and
  \[Sq^1\ :\ x_5=z_{15}\mapsto y_1y_2,\ \    x_6=z_{23}\mapsto y_1y_3 \  x_{7}=z_{27}\mapsto y_2y_3.\]
\end{lemma}

\begin{lemma} (\cite{Ya6})  
In $H^*(G/T)/(4)$, for all monomials $u\in P(y)^+/2$,
except for $y_{top}=y_1y_2y_3$, the elements $2u$ 
are written as elements in
$H^*(BT)$.  Moreover, in $BP^*(G/T)/I_{\infty}^2$,
there are $b_i\in BP^*(BT)$ such that
\[ b_2=2y_1,\ \ b_3=2y_2,\ \ b_4=2y_3,\ \ 
 b_6=2y_1y_3,\ \ b_7=2y_2y_3,\]
\[b_1=v_1y_1+v_2y_2+v_3y_3,\quad b_5=2y_1y_2+v_1y_3.\]
\end{lemma}
\begin{proof}  The last equations are given by
$ Q_1(x_1)=y_1,\   Q_2(x_1)=y_2,$
\[\  Q_3(x_1)=y_3,\
 Q_1(x_5)=y_3,\  Q_0(x_5)=y_1y_2.\]
\end{proof}

\begin{lemma} We have $t(E_7)_{(2)}=2^2$.
\end{lemma}
\begin{proof}
We get the result from $b_2b_7=(2y_1)(2y_2y_3)=2^2y_{top}$.
\end{proof}

\begin{cor}  Elements $b_1b_6, b_1b_7,b_2b_7$ are all
nonzero in $CH^*(R(\bG))/2$.
\end{cor}
\begin{proof}  Note that 
 $|y_1y_2y_3|=34$.
In $\Omega^*(\bar  \bF)/I_{\infty}^3$, we see
\[b_1b_5=2v_3y_{top},\quad  b_1b_6=2v_2y_{top},
\quad b_1b_7=2v_1y_{top}.\]
These elements are $\Omega^*$-module generators
in $Im(res_{k}^{\bar k}(\Omega^*(\bF)\to \Omega(\bar \bF))$
because $2y_1y_2y_3\not \in Im(res_k^{\bar k})$
from the fact $t(\bG)=2^2$.
\end{proof}

By Chevalley's theorem, $res_{K}$ is surjective.
Hence in $K^*(\bar R(\bG))/2$, 
\[ v_1^sy_2=b\in \bZ/2[b_1,...,b_{\ell}] \quad some \ s\in \bZ,\ i\ge 2.\]
Since $y_2\not\in Im(res_{K/2})$, we see $s\ge 1$.
The facts that $|b_i|\ge 6$ and $|y_2|=10$
imply $s=2$.
\begin{lemma}  We can take $b_2=2y_1+v_1^2y_2$
in $ k^*(G/T)/(v_1^3)$.
\end{lemma}

\begin{cor}
The restriction map $res_{K}$ is isomorphic, i.e.,
\[K^*(R(\bG))\cong K^*\otimes \Lambda_{\bZ}(b_1,b_2,b_5).\]
\end{cor}
\begin{proof}
Recall that $K^*(\bar R(\bG))\cong 
K^*\otimes P(y)\cong
 K^*\otimes \Lambda_{\bZ}(y_1,y_2,y_3).$
The corollary follows
from the
 $mod(2)$ restriction map $res_{K/2}$, which is given 
as $b_1\mapsto v_1y_1$,
$b_2\mapsto v_1^2y_2,$
$b_5\mapsto v_1y_3.$
\end{proof}

Here we give a direct proof (without using Chevalley's theorem) of the above lemma.
For this, we use $k(1)^*$ theory
(with the coefficient ring $ k(1)^*=k^*/p=\bZ/p[v_1]$)
of the Eilenberg-MacLane space $K(\bZ,3)$.
It is well known (for $p$ odd) 
\[ H^*(K(\bZ,3);\bZ/p)\cong
\bZ/p[y_1,y_2,...]\otimes \Lambda(x_1,x_2,...)\]
where $P^{p^{n-1}}...P^1(x_1)=x_{n}$ and $\beta x_{n+1}=y_n$.
When $p=2$, some  graded ring $grH^*(K(\bZ,3);\bZ/2)$ is isomorphic to the 
right hand side of the above isomorphism.
\begin{lemma} ( Theorem 3.4.4. (2) in \cite{Ya1})
We can take   $v_1^py_2=0$
in $grk(1)^*(K(\bZ,3))$.
\end{lemma}
\begin{proof}
We consider the AHss 
\[ E_2^{*,*'}\cong H^*(K(\bZ,3);\bZ/p)\otimes k(1)^*
\Longrightarrow k(1)^*(K(\bZ,3)).\]
It is known that all $y_n=Q_n(x_1)$ are permanent cycles.
The first nonzero differential is $d_{2p-1}=v_1\otimes Q_1$
\[ x_1\mapsto v_1y_1,\ \ x_2\mapsto 0,\ \ 
x_3\mapsto v_1y_1^p, \  x_4\mapsto v_1y_2^p\ ,...\]
(e.g., $Q_1(x_{i+2})=y_i^p$).
Hence $x_2$ and $y_1,y_2,...$ exist in $E_{2p}^{*,0}$.

On the other hand, from Hodgkin \cite{An-Ho}, it is well known
\[ k(1)^*(K(\bZ;n)[v_1^{-1}]=K(1)^*(K(\bZ;n))\cong
\bZ/p\quad for\ all \ n\ge 2.\]
Hence there is $s\ge 0$ such that $v_1^sy_2\in Im(d_r)$ for some
$r$.  By dimensional reason, we have 
$ d_{2(p-1)p+1}(x_2)=v_1^py_2.$
\end{proof}

\begin{proof}[Proof of Lemma 9.6.]
The element $x_1\in H^3(E_7;\bZ)$ represents
the map $i: E_7\to K(\bZ;3)$ such that 
$i^*(x_1)=x_1$ and $i^*(y_i)=y_i$.  Hence $v_1^2y_2=0$
in $grk(1)^*(E_7)$.

On the other hand, 
\[ k(1)^*(E_7)\cong k(1)^*(E_7/T)/(Ideal(S(t)^+)).\]
Hence we can write $v_1^2y_2=ty_1+t'$ for $t,t'\in S(t)^+.
$ Since $|v_1^2y_2|=|y_1|=6$, we see $t=0$.  Thus we get
$v_1^2y_2\in k(1)^*(BT)$.
\end{proof}

 \begin{prop}  
There is a filtration
whose associated graded ring is 
\[grK^*(R(\bG))  \cong  
K^*\otimes (B_1\oplus B_2)\quad where \]
\[ \begin{cases}
B_1=(P(b)/2\oplus 
   \bZ/2\{b_1b_7\}),\quad where\ P(b)=\Lambda_{\bZ}(b_1,b_2,b_5)   \\ 
  B_2=\bZ\{1,2b_1,b_3,b_4,b_1b_3,b_6,b_7,b_2b_7\}
\end{cases}\]
\end{prop}

\begin{proof}
Recall that
\[ K^*(R(\bG))\cong K^*\otimes P(b)
\ \cong K^*\otimes P(y)\cong K^*(\bar R\bG_k)).
\]
Let us write $\bar b_2=b_2-2v_1^{-1}b_1$ and 
$\bar b_5=b_5-v_1^{-1}b_1b_3$ so that
$\bar b_2=v_1^2y_2$, $\bar b_5=v_1y_3$ 
in $K^*(\bar R(\bG))$.  
Let us write $P(\bar b)=\Lambda(b_1,\bar b_2,\bar b_5)$.

We consider the graded ring
\[  gr'K^*(R(\bG))=K^*(R(\bG))/2\oplus 2K^*(R(\bG))
.\]
Here we have
\[K^*(R(\bG))\supset 2K^*\otimes P(b)= 2K^*\otimes
P(\bar b)\]
 \[=2K^*\{1,b_1,\bar b_2,..., \bar b_2\bar b_3,
b_1\bar b_2\bar b_3\}= 2K^*\{1,y_1,y_2,...,y_1y_2y_3\}\]
\[ =K^*\{2,2b_1,b_3,b_4,b_1b_3,b_6,b_7,b_1b_7\}.\]
Here we used 
$\ 2y_2=b_3,\ \  2y_3=b_4,\ \ 2v_1y_1y_2=b_1b_3,\ $ and
\[ 2y_1y_3=b_6,\ \ 2y_2y_3=b_7,\ \ 2v_1y_1y_2y_3=b_1b_7.\]
  The equation $2b_1b_7=4v_1y_1y_2y_3=v_1b_2b_7$ implies that
\[  grK^*(R(\bG))=gr'K^*(R(\bG))/(2b_1b_7)\oplus \bZ\{b_2b_7\},\]
which gives the graded ring in this lemma.
\end{proof}

We can not get $gr_{\gamma}(R(\bG))$ here, while we can see the following proposition as the case $(G,p)=(E_8,3)$.
\begin{prop}  Let $(G,p)=(E_7,2)$ and $\bG$ be versal.
Then there is an injection (as graded modules)
\[ B_1/(b_1b_2b_5)\oplus B_2/(2,2b_1) \subset 
 gr_{\gamma}(R(\bG))/2.\]
\end{prop}
\begin{proof}
We only prove that $b_1b_2$, $b_2b_5$ in $B_1$ are non zero in $CH^*(R(\bG))/2$.  (Other elements 
are more easily proven that they are also non zero in $CH^*(R(\bG))/2$ as the case $E_8$ $p=3$.)

We see
\[ b_1b_2=v_1y_1(2y_1+v_1^2y_2)=v_1^3y_1y_2\quad mod(2v_1I_{\infty},I_{\infty}^4).\]
Suppose $b_1b_2=v_1b$.  Then we see
$b=v_1^2y_1y_2$ $mod(2I_{\infty},v_1^3)$.
We can show that this does not happen, as the proof of
Lemma 9.10.  (For example, if $b=b_k\ mod(b_ib_j)$,
then $ b=b_3$ and this is a contradiction since
$b_3$ contains $2y_2$.)

For $b_2b_5$,  we see $b_2b_5=v_1^3y_2y_3$ $mod(2v_1I_{\infty}, I_{\infty}^4)$. We can see
$b_2b_5\not \in A$ as the case $b_1b_2$.
\end{proof} 

\section{The case $G=E_8$ and $p=2$}

It is known (\cite{Mi-Tod}) that
\[grH^*(E_8;\bZ/2)
\cong \bZ/2[y_1,y_2,y_3,y_4]/(y_1^8,y_2^4,y_3^2,y_4^2)\otimes \Lambda(x_1,...,x_8)\]
so that 
$ grH^*(E_7,;\bZ/2)\cong grH^*(E_8,\bZ/2)/(y_1^2,y_2^2,y_4,x_8).$

Here $Sq^2(x_7)=x_8$, $\beta(x_8)=y_4$.  Hence
\[ Q_1(x_7)=Q_0(x_8)=y_4,\quad |y_4|=30.\]
\begin{prop}
The restriction map $res_{K}$
is isomorphic. In fact,
\[K^*(R(\bG))/2\cong K^*/2[b_1,b_2,b_5,b_7]/
(b_1^8,b_2^4,b_5^2,b_7^2).\]
\end{prop}
\begin{proof}
From the case $E_7$, we have 
$b_2=v_1^2y_2$ $mod(2)$.  Moreover in $K^*(R(\bG))/2$,   we see
$ b_1=v_1y_1,$  $b_5=v_1y_3$,  $b_7=v_1y_4.$
\end{proof}
%Hence for all simply connected Lie
%groups $G$,  the restriction $res_K:
%K^*(\bF)\to K^*(\bar \bF)$ are  isomorphic.
%Thus we get the another proof of the Chevalley theorem.

\end{document}